\documentclass[12pt]{article}
\usepackage[english]{babel}
\usepackage[square,numbers]{natbib}
\usepackage{amscd,float}
\usepackage{latexsym,amsmath,amssymb,amsbsy, amsthm}

\usepackage[colorlinks = TRUE, linkcolor = blue, citecolor = blue,
urlcolor = blue]{hyperref}
\usepackage{setspace}
\usepackage{lscape,fancyhdr,fancybox}
\usepackage{graphicx,epsfig, xy}
\usepackage{color}
\usepackage[hmarginratio=1:1, vmarginratio =5:5,
textheight=22cm,bindingoffset=1.5cm, textwidth=14.6cm]{geometry}
\usepackage{float}
\usepackage{graphicx}
\usepackage{caption}
\usepackage{subcaption}

\newcommand{\comment}[1]{}

\numberwithin{equation}{section}

\newtheorem{theorem}{Theorem}[section]
\newtheorem{lemma}{Lemma}[section]
\newtheorem{proposition}{Proposition}[section]

\numberwithin{equation}{section}

\DeclareMathOperator{\Var}{Var}

\newtheorem{remark}{Remark}[section]

\theoremstyle{definition}
\newtheorem{definition}{Definition}[section]

\usepackage{amsfonts}

\newcommand{\Cov}{\text{Cov}}

\usepackage{txfonts,enumerate}

\DeclareMathOperator{\E}{E} 
\DeclareMathOperator{\cov}{Cov}

\DeclareMathOperator{\Tr}{Tr}

\newcommand{\beq}{\begin{eqnarray}}
\newcommand{\eeq}{\end{eqnarray}}
\newcommand{\ben}{\begin{eqnarray*}}
\newcommand{\een}{\end{eqnarray*}}

\title{Fluctuation of the free energy of Sherrington-Kirkpatrick model with Curie-Weiss interaction: the paramagnetic regime}
 \author{
\sc Debapratim Banerjee
 \\ \small Dept. of Statistics\\ University of Pennsylvania\\ dban@wharton.upenn.edu\\}
\begin{document}
 \maketitle
\begin{abstract}
We consider a spin system containing pure two spin Sherrington-Kirkpatrick Hamiltonian with Curie-Weiss interaction. The model where the spins are spherically symmetric was considered by \citet{Baiklee16} and \citet{Baikleewu18} which shows a two dimensional phase transition with respect to temperature and the coupling constant. In this paper we prove a result analogous to \citet{Baiklee16} in the ``paramagnetic regime" when the spins are i.i.d. Rademacher. We prove the free energy in this case is asymptotically Gaussian and can be approximated by a suitable linear spectral statistics. Unlike the spherical symmetric case the free energy here can not be written as a function of the eigenvalues of the corresponding interaction matrix. The method in this paper relies on a dense sub-graph conditioning technique introduced by \citet{Ban16}. The proof of the approximation by the linear spectral statistics part is close to \citet{Banerjee2017}.
\end{abstract}
\section{Introduction}

\subsection{The model description}
We at first give the description of the model. We start with a symmetric matrix $A=\left( A_{i,j} \right)_{i,j=1}^{n}$ where the entries in the strict upper triangular part of $A$ are i.i.d. standard Gaussian and for simplicity one might take $A_{i,i}=0$. The Hamiltonian corresponding to the Sherrington-Kirkpatrick model without any external field is given by 
\begin{equation}\label{SK:hamiltonian}
H_{n}^{SK}(\sigma):= \frac{1}{\sqrt{n}}\langle \sigma, A\sigma  \rangle =\frac{1}{\sqrt{n}} \sum_{i,j} A_{i,j} \sigma_{i} \sigma_{j}= \frac{2}{\sqrt{n}} \sum_{1\le i< j \le n} A_{i,j} \sigma_{i}\sigma_{j}.
\end{equation} 
Here $\sigma_{i}$'s are called spins and in this paper we shall only consider the case when $\sigma_{i} \in \{ -1, 1 \}$ for each $i$. In particular, one might consider the case when the spins $\sigma_{i}$'s are i.i.d. Rademacher random variables. This is known as the classical Sherrington- Kirkpatrick model. This model has got significant amount of interest in the study of spin glasses over the last few decades. Celebrated result like the proof of Parisi formula is considered one of the major advancements in this field. One might look at \citet{Panbook}, \citet{Tal05} for some information in this regard.

However the main focus of this paper is the following Hamiltonian 
\begin{equation}\label{our:hamiltonian}
H_{n}(\sigma):= H_{n}^{SK}(\sigma)+ H_{n}^{CW}(\sigma)
\end{equation}
where the Curie-Weiss Hamiltonian with coupling constant $J$ is defined by 
\begin{equation}\label{cw:hamiltonian}
H_{n}^{CW}(\sigma):= \frac{J}{n} \sum_{i,j=1}^{n} \sigma_{i}\sigma_{j} = \frac{J}{n}\left( \sum_{i=1}^{n} \sigma_{i} \right)^2.
\end{equation}
 Note that model corresponding to the Hamiltonian $H_{n}^{CW}(\sigma)$ is a simple model which can be studied in considerable details. One might look at \citet{54} for a reference. 
 
This Hamiltonian in \eqref{our:hamiltonian} was defined in \citet{Talbook}(see (4.22)) as a model which have difficulties of the ferromagnetic interactions however with a familiar disorder. This model was introduced as a prelude to study the Hopfield model in \citet{Talbook}. 
  

The main result of this paper is  a limit theorem for the free energy corresponding to the Hamiltonian $H_{n}(\sigma)$ when $\beta< \frac{1}{2}$ and $\beta J < \frac{1}{2}$ whenever $\sigma_{i}$'s are i.i.d. Rademacher variables. If the spins $\mathbf{\sigma}=(\sigma_{1},\ldots,\sigma_{n})$ are distributed according to the uniform measure on the sphere $S_{n-1}$ where $S_{n-1}:= \left\{ \sigma \in \mathbb{R}^{n}~|~ ||\sigma||^2=n \right\}$, then the analogous Hamiltonian was considered in \citet{Baiklee16} and \citet{Baikleewu18}. However the results in \citet{Baiklee16} are much more general than the current paper in the sense they are able to consider any $\beta>0,J>0$. Depending on the values of $\beta, J$, there are three distinct regimes where the free energy shows different behaviors. In particular, the regime $\beta< \frac{1}{2}$ and $\beta J < \frac{1}{2}$ is known as the para-magnetic regime where the result analogous to this paper was obtained in \citet{Baiklee16}. The regime when $\beta> \frac{1}{2}$ and $J<1$ is known as the spin glass regime and the other case ($\beta J > \frac{1}{2} $ and $J>1$) is known as the ferromagnetic regime. Although the results in \citet{Baiklee16} are much more general than the current paper in terms of possible choices of $(\beta,J)$, the technique of that paper is restricted to the case when the  spins $\mathbf{\sigma}=(\sigma_{1},\ldots,\sigma_{n})$ are distributed according to the uniform measure on the sphere $S_{n-1}$ which does not cover the case when $\sigma_{i}$'s are i.i.d. Rademacher random variables. This is the problem we consider in this paper. 

We now give a very brief overview of the literature for the fluctuation of free energy of classical Sherrington-Kirkpatrick model in with absence of any external field. 

The classical Sherrington-Kirkpatrick model with no external field ($h=0$) under goes a phase transition at $\beta= \frac{1}{2}$.  When the spins $\sigma_{i}$'s are i.i.d. Rademacher and $\beta< \frac{1}{2}$ the free energy has a Gaussian limiting distribution. One might look at \citet{ALR87} and \citet{CoNe95} for some references. The case $\beta> \frac{1}{2}$ is known as the low temperature regime. To the best of our limited knowledge, very few things are known about the fluctuations of the free energy in this regime. One might look at \citet{Cha17} where it is proved that the fluctuation of the free energy of the Sherrington-Kirkpatrick model is at least $O(1)$. When the spins are uniformly distributed on $S_{n-1}$, the free energy analogously undergoes a phase transition at $\beta=\frac{1}{2}$. When $\beta< \frac{1}{2}$, the free energy has a Gaussian limiting distribution and can be approximated by a linear spectral statistics of the eigenvalues. The low temperature case ($\beta> \frac{1}{2}$) is also well-known. Here the free energy has a limiting GOE Tracy-Widom distribution with $O\left( n^{-\frac{2}{3}} \right)$ fluctuations. One might look at \citet{Baikleestat} for a reference.

\noindent 
 Finally, The model considered in this paper (Hamiltonian defined in \eqref{our:hamiltonian}) was also considered in \citet{Chen12} from the point of view of thermodynamic limit of the free energy. One might also look at \citet{CadelR} where this model was studied using cavity method. However to the best of our limited knowledge the problem of fluctuation of free energy remained open.  

\subsection{Preliminary definitions}
We now give some preliminary definitions.
We start with defining a Hamiltonian which generalizes the one defined in \eqref{our:hamiltonian}. 
\begin{definition}(interactions)\label{def:ourham}
Suppose $A_{i,j}$, $1\le i \le j \le n$ be i.i.d. standard Gaussian random variables. Set $A_{j,i}=A_{i,j}$ for $i<j$. Let $M_{i,j}= \frac{1}{\sqrt{n}}A_{i,j}+ \frac{J}{n}$ and $M_{i,i}= \frac{1}{\sqrt{n}}A_{i,i}+ \frac{J'}{n} $ for some $n$ independent non negative fixed constants $J$ and $J'$. One considers the Hamiltonian $H_{n}(\sigma)= \langle \sigma, M\sigma \rangle$. The defined Hamiltonian is more general than the one defined in \eqref{our:hamiltonian} in the following sense. Definition \ref{def:ourham} allows $A_{i,i}$ to be non zero random variables with $J'$ being any arbitrary constant.
\end{definition}
\begin{definition}(Partition function and Free energy)     
Given any Hamiltonian $H_{n}(\sigma)$ where $\sigma=(\sigma_{1},\ldots, \sigma_{n})$ are distributed according to a measure $\Psi_{n}$, the partition function and free energy at an inverse temperature $\beta$ are denoted by $Z_{n}(\beta)$ and $F_{n}(\beta)$ respectively and defined as follows. 
\begin{equation}\label{eq:partition}
Z_{n}(\beta) := \int \exp\left\{\beta H_{n}(\sigma)  \right\} d \Psi_{n}(\sigma)
\end{equation} 
and 
\begin{equation}
F_{n}(\beta) : =\frac{1}{n} \log\left(Z_{n}(\beta)\right).
\end{equation}
In our case we take $\Psi_{n}$ to be the uniform probability measure on the Hypercube $\{-1, +1  \}^n$.
\end{definition}
 In our case $Z_{n}(\beta)$ is as follows:
\begin{equation}
\begin{split}
Z_{n}(\beta)&= \sum_{\sigma \in \{ -1, +1 \}^{n}}\frac{1}{2^{n}}\exp\left\{ \sum_{i,j=1}^{n}\frac{\beta}{\sqrt{n}} A_{i,j}\sigma_{i}\sigma_{j} + \frac{\beta J}{n} \sum_{i,j=1}^{n} \sigma_{i}\sigma_{j}+ \beta(J'-J)\right\}\\
&=  \sum_{\sigma \in \{ -1, +1 \}^{n}}\frac{1}{2^{n}} \exp\left\{ \frac{2\beta}{\sqrt{n}}\sum_{1\le i<j\le n}^{n}A_{i,j}\sigma_{i}\sigma_{j} + \frac{2\beta J}{n} \sum_{1\le i<j\le n}\sigma_{i}\sigma_{j} + \frac{\beta}{\sqrt{n}}\sum_{i=1}^{n}\left(A_{i,i}+\frac{J'}{\sqrt{n}}\right)\right\}.
\end{split}
\end{equation} 
Finally we define the Wasserstein distance between two distribution functions. This distance is crucially used at many places of the proofs.
\begin{definition}\label{Wass}
We at first fix $p \ge 1$.
Suppose $F^{1}$ and $F^{2}$ are two distribution functions such that $\int_{x \in \mathbb{R}} |x|^{p}d F^{1}(x)< \infty$ and $\int_{x \in \mathbb{R}} |x|^{p}dF^{2}(x)< \infty$. Then the Wasserstein distance for $p$ between $F^{1}$ and $F^{2}$ is is denoted by $W_{p}$ and defined to be 
\begin{equation}
W_{p}\left( F^{1}, F^{2} \right):=  \left[\inf_{X \sim F^{1}; Y \sim F^{2}}\E\left[ \left| X- Y \right|^{p} \right]\right]^{\frac{1}{p}}.
\end{equation}
\end{definition}
Observe that the Wasserstein distance is defined for two distribution functions. However when we write $\left[\inf_{X \sim F^{1}; Y \sim F^{2}}\E\left[ \left| X- Y \right|^{p} \right]\right]^{\frac{1}{p}}$, we consider two random variables $X\sim F^{1}$ and $Y \sim F^{2}$ such that $X$ and $Y$ are defined on the same measure space such that $\E\left[ \left| X- Y \right|^{p}\right]$ takes the lowest possible value. 

\noindent  
The following result on the Wasserstein distance is well known.
\begin{proposition}\label{prop:wass}
Suppose $\left\{X_{n}\right\}_{n=1}^{\infty}$ be a sequence of random variables and $X$ be a random variable. If $W_{2}\left( F^{X_{n}}, F^{X}\right) \to 0$, then $X_{n} \stackrel{d}{\to} X$ and $\E[X_{n}^2] \to \E[X^2]$.
\end{proposition}
One might see \citet{Mal72} for a reference.
\section{Main result}
We are ready to state the main result of this paper.
\begin{theorem}\label{thm:asymptotic}

\noindent 
\begin{enumerate}
\item (Asymptotic normality)
Consider the Hamiltonian $H_{n}(\sigma)$ as defined in Definition \ref{def:ourham}. Let $F_{n}(\beta)$ be the free energy corresponding to the Hamiltonian $H_{n}(\sigma)$. When $\beta < \frac{1}{2}$ and $\beta J< \frac{1}{2}$ the following result holds:
\begin{equation}
n\left(F_{n}(\beta)- F(\beta)\right) \stackrel{d}{\to} N(f_{1},\alpha_{1}) 
\end{equation} 
where $F(\beta)= \beta^2$, 
\begin{equation}
\alpha_{1}= -\beta^2- \frac{1}{2}\log\left(1- 4\beta^2\right)
\end{equation}
and
\begin{equation}
f_{1}=  -\frac{1}{2} \log\left(1- 2\beta J\right) + \beta(J'- J) + \frac{1}{4}\log\left(1- 4\beta^2\right).
\end{equation}
\item (Approximation by signed cycle counts) For any sequence $m_{n}$ diverging to infinity such that $m_{n}=o\left( \sqrt{\log n} \right)$,  one also has the following approximation result for the log partition function $\log\left(Z_{n}(\beta)\right)$.
\begin{equation}
\begin{split}
&\log\left(Z_{n}(\beta)\right) + \frac{1}{2} \log\left(1- 2\beta J\right)- (n-1)\beta^2 + \beta (J  -J') - \beta C_{n,1} -\\
&~~~~~~~~~~~~~~~~~~~~~~~~~~~~ \sum_{k=2}^{m_{n}}\frac{2(2\beta)^{k}\left(C_{n,k}- (n-1)\mathbb{I}_{k=2}\right)- (2\beta)^{2k}}{4k} \stackrel{p}{\to} 0.
\end{split}
\end{equation}
 Here the statistics $C_{n,k}$'s are taken according to Definition \ref{def:signedcycles}. 
\end{enumerate}
\end{theorem}
Our next result is Theorem \ref{thm:asymptotic} where the approximations of $C_{n,k}$'s by linear spectral statistics are stated. 
Before going to Theorem \ref{Thm:approximation}, we need some important definitions.
We now introduce an important generating function. 
Given any $r \in \mathbb{N}$, let
\begin{equation}\label{gen:f}
\left( \frac{1-\sqrt{1-4z^2}}{2z} \right)^r= \sum_{m=r}^{\infty} f(m,r) z^{m}.
\end{equation}
The coefficients $f(m,r)$'s are key quantities for defining the variances and covariances of linear spectral statistics constructed from different power functions. 
For any $k \in \mathbb{N}$ denote 
\begin{equation}\label{def:catalan}
\psi_{k}=\left\{  
\begin{array}{ll}
0 & \text{if $k$ is odd}\\
\frac{1}{\frac{k}{2}+1} \binom{k}{\frac{k}{2}} & \text{if $k$ is even}.
\end{array}
\right.
\end{equation}
So $\psi_k$ is the $\frac{k}{2}$-th Catalan number for every even $k$.
Finally, we define a set of rescaled Chebyshev polynomials. 
These polynomials are important for drawing the connection between signed cycles $C_{n,k}$'s and the spectrum of adjacency matrix. 
The standard Chebyshev polynomial of degree $m$ is denoted by $S_{m}(x)$ and can be defined by the identity
\begin{equation}\label{def:ChebyshevI}
S_{m}\left(\cos(\theta)\right) = \cos(m\theta).
\end{equation}
In this paper we use a slight variant of $S_m$, denoted by $P_{m}$ and defined as
\begin{equation}\label{def:ChebyshevII}
P_{m}(x)= 2S_{m}\left(\frac{x}{2}\right).
\end{equation}
In particular, 
$P_{m}(2\cos(\theta)) = 2 \cos(m\theta)$.
It is easy to note that 
$P_{m}\left( z+ z^{-1}\right)= z^{m} + z^{-m}$ for all $z \in \mathbb{C}$.
One also notes that $P_{m}(\cdot)$ is even and odd whenever $m$ is even or odd respectively.

\begin{theorem}\label{Thm:approximation}(Approximation of cycles by linear spectral statistics)
Let $\tilde{A}$ be the matrix obtained by putting $0$ on the diagonal of the matrix $A$.
Let $P_{k}$ be as defined in \eqref{def:ChebyshevII}
. Then to following is true for any $3\le k= o\left(  \sqrt{\log n}\right)$ under $\mathbb{P}_{n}$.
\begin{equation}
C_{n,k}- \left\{\Tr \left( P_{k}\left( \frac{1}{\sqrt{n}} \tilde{A} \right) \right)- \E\left[ \Tr \left( P_{k}\left( \frac{1}{\sqrt{n}}\tilde{A} \right) \right)  \right]\right\} \stackrel{p}{\to} 0.
\end{equation}
Here for any function $f$ and a matrix $A$ 
\begin{equation}
\Tr\left[ f(A) \right]=\sum_{i=1}^{n} f(\lambda_{i})
\end{equation}
where $\lambda_{1},\ldots,\lambda_{n}$ are the eigenvalues of the matrix $A$. 
\end{theorem}
The proof of Theorem \ref{Thm:approximation} is given in Section \ref{sec:thmapp}.
\section{Proof techniques and related definitions} 
As mentioned earlier, the fundamental technique of the proof of Theorem \ref{thm:asymptotic} is completely different from that of \citet{Baiklee16}. The proof in the current paper is based on the dense sub graph conditioning technique introduced in \citet{Ban16}. The fundamental idea is to view the free energy as the log of the Radon-Nikodym derivative $\left(\log\frac{d\mathbb{Q}_{n}}{d\mathbb{P}_{n}}\right)$ of two suitably defined sequences of measures $\mathbb{P}_{n}$ and $\mathbb{Q}_{n}$. Now one introduces a class of random variables called the signed cycles (Definition \ref{def:signedcycles}) and prove that these variables asymptotically determine the full Radon-Nikodym derivative. This is done by a fine second moment argument. The argument in this part is highly motivated from a paper by \citet{Jan} where it is proved that a similar kind of argument holds for random regular graphs where the signed cycle counts are replaced by standard cycle counts. The technique of cycle conditioning was also used in \citet{MNS12} in their proof of contiguity of the probability measures induced by a planted partition model and the Erd\H{o}s- R\'enyi model in the sparse regime. 

We now start with defining the signed cycles random variables. 
\begin{definition}\label{def:signedcycles}
Let $A$ be a $n \times n$ symmetric matrix with with the strict upper triangular part being i.i.d. mean $0$ and variance $1$. For $k\ge 2$, we define the signed cycles random variables $C_{n,k}$ as follows:
\begin{equation}
C_{n,k}:= \left( \frac{1}{\sqrt{n}} \right)^{k}\sum_{i_0,i_1,\ldots,i_{k-1}}A_{i_{0},i_{1}} A_{i_{1},i_{2}} \ldots A_{i_{k-1},i_{0}}.
\end{equation}
Here $i_{0},\ldots,i_{k-1}$ are taken to be all distinct. For $k=1$, $C_{n,k}$ is simply defined as follows:
\begin{equation}
C_{n,1} := \left( \frac{1}{\sqrt{n}} \right) \sum_{i} A_{i,i}.
\end{equation}

\end{definition}
In this paper we require the concept of mutual contiguity of two sequences of measures heavily. Now we define these concepts. If someone is interested one might have a look at \citet{LeCam} and \citet{LeCam00} for general discussions on contiguity.

\begin{definition}(Contiguity)
For two sequences of probability measures $\left\{\mathbb{P}_n\right\}_{n=1}^{\infty}$ and $\left\{\mathbb{Q}_n\right\}_{n=1}^{\infty}$ defined on $\sigma$-fields $(\Omega_n,\mathcal{F}_n)$,
we say that $\mathbb{Q}_n$ is contiguous with respect to $\mathbb{P}_n$, denoted by $\mathbb{Q}_n \triangleleft \mathbb{P}_n$, if for any event sequence $A_n$, $\mathbb{P}_n(A_n)\to 0$ implies $\mathbb{Q}_n(A_n)\to 0$.
We say that they are (asymptotically) mutually contiguous, denoted by $\mathbb{P}_n \triangleleft\triangleright \mathbb{Q}_n$, if both $\mathbb{Q}_n\triangleleft \mathbb{P}_n$ and $\mathbb{P}_n\triangleleft \mathbb{Q}_n$ hold.
\end{definition}

\noindent
The following result gives an useful way to study mutual contiguity:
\begin{proposition}\label{prop:useI}
Suppose that $L_n=\frac{d\mathbb{Q}_n}{d\mathbb{P}_n}$, regarded as a random variable on $(\Omega_n,\mathcal{F}_n,\mathbb{P}_n)$, converges in distribution to some random variable $L$ as $n \to \infty$. Then $\mathbb{P}_n$ and $\mathbb{Q}_n$
are mutually contiguous if and only if $L > 0$ a.s. and $\E [L] = 1$.
\end{proposition}
One might look at Proposition 3 of \citet{Jan} for a proof.

\noindent 
We now state a result on mutual contiguity of measures.
\begin{proposition}
\label{prop:norcont}(Janson's second moment method):
Let $\mathbb{P}_n$ and $\mathbb{Q}_n$ be two sequences of probability measures such that for each $n$, both are defined on the common $\sigma$-algebra $(\Omega_n, \mathcal{F}_n)$.
Suppose that for each $i\geq 1$, $W_{n,i}$ are random variables defined on $(\Omega_n,\mathcal{F}_n)$.
Then the probability measures $\mathbb{P}_n$ and $\mathbb{Q}_n$ are asymptotically mutually contiguous if the following conditions hold simultaneously:
\begin{enumerate}[(i)]
\item $\mathbb{Q}_n$ is absolutely continuous with respect to $\mathbb{P}_n$
for each $n$;
\item The likelihood ratio statistic $Y_n = \frac{\mathrm{d}\mathbb{Q}_n}{\mathrm{d}\mathbb{P}_n}$ satisfies 
\begin{equation}
	\label{eq:lr-square}
\limsup_{n\to\infty}\E_{\mathbb{P}_n}\left[Y_n^2\right] \leq
\exp\left\{\sum_{i=1}^\infty \frac{\mu_i^2}{\sigma_i^2}\right\} < \infty.
\end{equation}
\item For any fixed $k\ge 1$, one has $\left( W_{n,1},\ldots, W_{n,k} \right)|\mathbb{P}_{n} \stackrel{d}{\to} \left(Z_{1}, \ldots, Z_{k}\right) $  and $\left( W_{n,1},\ldots, W_{n,k} \right)|\mathbb{Q}_{n} \stackrel{d}{\to} \left(Z'_{1}, \ldots, Z'_{k}\right)$. Further
$Z_{i}\sim N(0,\sigma_{i}^{2})$ and $Z'_{i} \sim N(\mu_{i},\sigma_{i}^{2})$ are sequences of independent random variables.

\item Under $\mathbb{P}_{n}$, $W_{n,i}$'s are uncorrelated and there exists a sequence $m_{n} \to \infty$ such that 
\begin{equation}
\Var\left[ \sum_{i=1}^{m_{n}} \frac{\mu_{i}}{\sigma_{i}^2} W_{n,i} \right] \to C < \infty
\end{equation}
Here the $\Var$ is considered with respect to the measure $\mathbb{P}_{n}$.
\end{enumerate}
In addition, we have that
under $\mathbb{P}_n$,
\begin{equation}
	\label{eq:lr-limit}
Y_n \stackrel{d}{\to} \exp\left\{\sum_{i=1}^\infty \frac{\mu_i Z_i - \frac{1}{2}\mu_i^2}{\sigma_i^2}\right\}.
\end{equation}
Furthermore, given any $\epsilon,\delta>0$ there exists a natural number $K=K(\delta,\epsilon)$ such that for any sequence $n_l$ there is a further subsequence $n_{l_m}$ such that 
\begin{equation}
	\label{eq:janson-decomp}
\limsup_{m\to\infty} \mathbb{P}_{n_{l_m}}\left( \left| \log(Y_{n_{l_m}}) - 
\sum_{k=1}^{K} \frac{2\mu_{k} W_{n_{l_m},k} -\mu_{k}^{2}}{2\sigma_{k}^{2}} 
\right|
\ge \epsilon \right) \le  \delta.
\end{equation}
\end{proposition}

\noindent 

Proposition \ref{prop:norcont} is one of the most important results required for the proof of Theorem \ref{thm:asymptotic}. In particular, the rest of the proof relies on defining the measures $\mathbb{P}_{n}$ and $\mathbb{Q}_{n}$ and $W_{n,i}$'s properly. It is worth noting that in this context the statistics $C_{n,i}$'s serve as $W_{n,i}$'s. 

\noindent 
We now give the proof of Proposition \ref{prop:norcont}.

\noindent
\textbf{Proof of Proposition \ref{prop:norcont}:}
\paragraph{Proof of mutual contiguity and \eqref{eq:lr-limit}}
This proof is broken into two steps.
We focus on proving \eqref{eq:lr-limit}.
Given \eqref{eq:lr-limit}, mutual contiguity is a direct consequence of Proposition \ref{prop:useI}.

\medskip

\textbf{Step 1.}
We first prove the random variable on the right hand side of \eqref{eq:lr-limit} is almost surely positive and has mean $1$. 
Let us define 
\begin{equation}
L:=\exp\left\{ \sum_{i=1}^{\infty}\frac{2\mu_{i}Z_{i}-\mu_{i}^{2}}{2\sigma_{i}^{2}} \right\},\qquad
L^{(m)}:= \exp\left\{ \sum_{i=1}^{m}\frac{2\mu_{i}Z_{i}-\mu_{i}^{2}}{2\sigma_{i}^{2}} \right\},\quad
 m\in \mathbb{N}.
\end{equation}
As $Z_i \sim N(0,\sigma_{i}^{2})$, for any $i\in\mathbb{N}$, and so
\begin{equation}
\E\left[{\exp}
\left\{\frac{2\mu_{i}Z_{i}-\mu_{i}^{2}}{2\sigma_{i}^{2}} \right\}\right]=1.
\end{equation}
So $\{L^{(m)}\}_{m=1}^{\infty}$ is a martingale sequence and 
\begin{equation}
\E\left[ \big(L^{(m)}\big)^2 \right]=\prod_{i=1}^{m} \exp\left\{ \frac{\mu_{i}^{2}}{\sigma_{i}^{2}} \right\}=\exp\left\{ \sum_{i=1}^{m} \frac{\mu_{i}^{2}}{\sigma_{i}^{2}} \right\}.
\end{equation}
Now by the righthand side of \eqref{eq:lr-square},
$L^{(m)}$ is a $L^2$ bounded martingale.
Hence, $L$ is a well defined random variable with
\begin{equation}
\E[L] = 1,\qquad
\E[L^2]= \exp\left\{ \sum_{i=1}^{\infty} \frac{\mu_{i}^{2}}{\sigma_{i}^{2}} \right\}.
\end{equation}
On the other hand $\log(L)$ is a limit of Gaussian random variables, hence $\log(L)$ is Gaussian with
\begin{equation}
\E[\log(L)]= -\frac{1}{2} \sum_{i=1}^{\infty} \frac{\mu_{i}^{2}}{\sigma_{i}^{2}}, \qquad 
\Var( \log(L) ) = \sum_{i=1}^{\infty} \frac{\mu_{i}^{2}}{\sigma_{i}^{2}}.
\end{equation}
Hence $\mathbb{P}(L=0)= \mathbb{P}(\log(L)=-\infty)=0$. 

\medskip

\textbf{Step 2.}
Now we prove $Y_n \stackrel{d}{\to} L$.
Since
\begin{equation}
\limsup_{n \to \infty}\E_{\mathbb{P}_n}\left[ Y_n^2\right]<\infty,
\end{equation}
condition (iv) implies that the sequence $Y_n$ is tight. 
Prokhorov's theorem further implies that there is a subsequence 
$\{ n_{k} \}_{k=1}^{\infty}$ such that $Y_{n_k}$ converge in distribution to some random variable $L(\{ n_{k} \})$. 
In what follows, we prove that the distribution of $L(\{ n_{k} \})$ does not depend on the subsequence $\{ n_{k} \}$. In particular, $L(\{ n_{k} \})\stackrel{d}{=} L$.
To start with, note that since $Y_{n_k}$ converges in distribution to $L(\{ n_{k} \})$, for any further subsequence $\{ n_{k_l} \}$ of $\{ n_{k}\}$, $Y_{n_{k_l}}$ also converges in distribution to $L(\{ n_{k} \})$.

Given any fixed $\epsilon>0$ take $m$ large enough such that 
\begin{equation}
\exp\left\{ \sum_{i=1}^{\infty} \frac{\mu_{i}^{2}}{\sigma_{i}^{2}} \right\}-\exp\left\{ \sum_{i=1}^{m} \frac{\mu_{i}^{2}}{\sigma_{i}^{2}} \right\} < \epsilon.
\end{equation}
For this fixed number $m$, consider the joint distribution of 
$(Y_{n_{k}},W_{n_k,1},\ldots,W_{n_k,m})$. 
This sequence of $m+1$ dimensional random vectors with respect to $\mathbb{P}_{n_k}$ is tight by condition (ii). 
So it has a further subsequence such that $$(Y_{n_{k_l}},W_{n_{k_l},1},\ldots,W_{n_{k_l},m})|\mathbb{P}_{n_{k_l}}\stackrel{d}{\to}\left((H_1,\ldots,H_{m+1})\in (\Omega(\{ n_{k_l} \}),\mathcal{F}(\{ n_{k_l} \}),P(\{ n_{k_l} \}))(say).\right).$$ 
where $H_{1}\stackrel{d}{=} L(\{ n_{k} \})$ and $\left(H_{2},\ldots, H_{m+1}\right)\stackrel{d}{=}\left(  Z_{1},\ldots,Z_{m}\right)$
We are to show that we can define the random variables $L^{(m)}$ and $L(\{ n_{k} \})$ in such a way that there exist suitable $\sigma$-algebras $\mathcal{F}_1 \subset\mathcal{F}_2$ such that 
$L^{(m)} \in \mathcal{F}_1$, 
$L(\{ n_{k} \}) \in \mathcal{F}_2$,
and $\E \left[ L(\{ n_{k} \})\left|\right. \mathcal{F}_{1} \right]= L^{(m)}$.

Since $\limsup_{n \to \infty}\E_{\mathbb{P}_n}\left[ Y_n^2\right]< \infty$, the sequence $Y_{n_{k_l}}$ is uniformly integrable. 
This, together with condition (i), leads to
\begin{equation}\label{eqn_expder}
\E[L(\{ n_{k} \})] = \lim_{l\to\infty} 
\E_{\mathbb{P}_{n_{k_l}}} [ Y_{n_{k_l}} ] = 1.
\end{equation}
Now take any positive bounded continuous function $f:\mathbb{R}^m \to \mathbb{R}$. By Fatou's lemma 
\begin{equation}
	\label{eqn_ineq}
\liminf_{l\to\infty} 
\E_{\mathbb{P}_{n_{k_l}}} \left[f (W_{n_{k_l},1},\ldots, W_{n_{k_l},m} )Y_{n_{k_l}} \right] \ge \E \left[ f\left(Z_1,\ldots,Z_{m}\right)L(\{ n_{k} \})\right].
\end{equation}
However for any constant $\xi$, \eqref{eqn_expder} implies
$\xi=\xi\E_{\mathbb{P}_{n_{k_l}}}[ Y_{n_{k_l}} ] \to \xi\E[L(\{ n_{k} \})]= \xi$.
Observe that given any bounded continuous function $f$ we can find $\xi$ large enough so that $f+ \xi$ is a positive bounded continuous function.
So \eqref{eqn_ineq} is indeed implied by Fatou's lemma.

Now 
\begin{equation}
\begin{split}
&\liminf \E_{\mathbb{P}_{n_{k_l}}} \left[\left(f (W_{n_{k_l},1},\ldots,W_{n_{k_l},m} )+ \xi\right)Y_{n_{k_l}} \right]\\
&=\liminf \E_{\mathbb{P}_{n_{k_l}}}\left[ f(W_{n_{k_l},1},\ldots,W_{n_{k_l},m} ) Y_{n_{k_l}} \right] +\xi\\
& \ge  \E\left[ \left(f(Z_{1},\ldots,Z_{m} )+\xi\right) L(\{ n_{k} \}) \right] 
\end{split}
\end{equation}
So \eqref{eqn_ineq} holds for any bounded continuous function $f$. 
On the other hand, replacing $f$ by $-f$ we have 
\begin{equation}
	\label{eqn_ineqII}
\lim_{l\to\infty} 
\E_{\mathbb{P}_{n_{k_l}}}
\left[f(W_{n_{k_l},1},\ldots,W_{n_{k_l},m})Y_{n_{k_l}} \right] 
= \E \left[ f(Z_1,\ldots,Z_{m})L(\{ n_{k} \})\right].
\end{equation}
Now condition (ii) leads to
\begin{equation}
\int f(W_{n_{k_l},1},\ldots,W_{n_{k_l},m})Y_{n_{k_l}} \mathrm{d}\mathbb{P}_{n_{k_l}}= \int f(W_{n_{k_l},1},\ldots,W_{n_{k_l},m})\mathrm{d}\mathbb{Q}_{n_{k_l}} \to \int f(Z_1',\ldots,Z_{m}') \mathrm{d}Q.
\end{equation}
Here  $Q$ is the measure induced by $(Z_1',\ldots,Z_{m}')$. In particular, one can take the measure $Q$ such that
$(Z_1,\ldots, Z_{m})$ themselves are distributed as $(Z_1',\ldots,Z_{m}')$ under the measure $Q$. 
This is true since 
\begin{equation}
\int f(Z_1',\ldots, Z_{m}') \mathrm{d}Q= \E\left[ f(Z_1,\ldots, Z_{m}) L^{(m)}\right].
\end{equation}
for any bounded continuous function $f$, and so
$\int_{A}  \mathrm{d}Q= \E[ \mathbf{1}_{A} L^{(m)} ]$
for any $A \in \sigma(Z_1,\ldots, Z_{m})$. 
Now looking back into \eqref{eqn_ineqII}, we have for any $A \in \sigma(Z_1,\ldots, Z_{m})$,
$\E[ \mathbf{1}_{A} L^{(m)} ]= \E\left[ \mathbf{1}_{A} L(\{ n_{k} \}) \right]$.
Since by definition $L^{(m)}$ is $\sigma(Z_1,\ldots, Z_{m})$ measurable, we have 
\begin{equation}
L^{(m)}=\E\left[ L(\{ n_{k} \}) \left|\right. \sigma(Z_1,\ldots, Z_{m})  \right].	
\end{equation}

From Fatou's lemma 
\begin{equation}
\E[ L(\{ n_{k} \})^2]\le \liminf_{n \to \infty} \E_{\mathbb{P}_n}[Y_n^2]= \exp\left\{ \sum_{i=1}^{\infty} \frac{\mu_i^{2}}{\sigma_{i}^{2}} \right\}.
\end{equation}
As a consequence, we have 
\begin{equation}
0 \le \E|L(\{ n_{k} \})-L^{(m)}|^2 = \E[L(\{ n_{k} \})^2]-\E[L^{(m)2}]< \epsilon.
\end{equation}
So $W_2(F^{L^{(m)}},F^{L(\{ n_{k} \})})< \sqrt{\epsilon}$. Here $F^{L^{(m)}}$ and $F^{L(\{ n_{k} \})}$ denote the distribution functions corresponding to $L^{(m)}$ and $L(\{ n_{k} \})$ respectively. As a consequence, $W_2(F^{L^{(m)}},F^{L(\{ n_{k} \})}) \to 0$ as $m \to \infty.$ Hence $L^{(m)} \stackrel{d}{\to} L(\{ n_{k} \})$ by the result stated after Definition \ref{Wass}. 
On the other hand, we have already proved $L^{(m)}$ converges to $L$ in $L^2$. So $L(\{ n_{k} \})\stackrel{d}{=}L$. 
\paragraph{Proof of \eqref{eq:janson-decomp}}
We start with a sub sequence $\{ n_{l} \}$. We shall choose $k$ large enough which shall be specified later.
We also know that both the random variables $\log\left(Y_{n_{l}}\right)$ and $ \left\{ \sum_{i=1}^{k} \frac{2\mu_{i} W_{n_{l},i}- \mu_{i}^2}{2\sigma_{i}^2} \right\}$ are tight. 

We now prove that there is a $M$ invariant of $k$ such that both the probabilities 
\begin{equation}
\begin{split}
&\mathbb{P}_{n_{l}}\left[ -M \le \log\left(Y_{n_{l}}\right) \le M \right] \ge 1- \frac{\delta}{100}\\
&\mathbb{P}_{n_{l}} \left[  -M \le \left\{ \sum_{i=1}^{k} \frac{2\mu_{i}W_{n_{l},i}- \mu_{i}^2 }{2 \sigma_{i}^2} \right\} \le M  \right] \ge 1 -\frac{\delta}{100}
\end{split}
\end{equation}
for all $n_{l}$.
Since the random variable $Y_{n_{l}}$ do not depend on $k$ the first inequality is obvious. For the second inequality observe that 
\begin{equation}
\Var\left[ \left\{ \sum_{i=1}^{k} \frac{2\mu_{i}W_{n_{l},i}- \mu_{i}^2 }{2 \sigma_{i}^2} \right\} \right]\le \Var \left[ \sum_{i=1}^{m_{n}} \frac{2\mu_{i}W_{n_{l},i}- \mu_{i}^2 }{2 \sigma_{i}^2} \right]
\end{equation}
where $m_{n}$ is a sequence increasing to infinity as mentioned in Proposition \ref{prop:norcont}. Now 
\begin{equation}
\begin{split}
& \Var \left[ \sum_{i=1}^{m_{n}} \frac{2\mu_{i}W_{n_{l},i}- \mu_{i}^2 }{2 \sigma_{i}^2} \right] < C'
\end{split}
\end{equation}
for all $n_{l}$.
for a deterministic constant $C'$. As a consequence,
\begin{equation}
\begin{split}
\mathbb{P}_{n_{l}}\left[ \left| \sum_{i=1}^{k} \frac{2\mu_{i}W_{n_{l},i}- \mu_{i}^2 }{2 \sigma_{i}^2} \right|> M \right]\le \frac{C'}{M^2}\le \frac{\delta}{100}
\end{split}
\end{equation}
where $M^2= \frac{100C'}{\delta}$.
\begin{equation}
\mathbb{P}_{n_{l}}\left[ -M \le \log\left(Y_{n_{l}}\right) \le M \cap  -M \le \left\{ \sum_{i=1}^{k} \frac{2\mu_{i}W_{n_{l},i}- \mu_{i}^2 }{2 \sigma_{i}^2} \right\} \le M   \right] \ge 1- \frac{\delta}{50}.
\end{equation}
Now $\log(\cdot)$ is an uniformly continuous function on $[e^{-M}, e^{M}]$. So given $\epsilon>0$, there exists $\tilde{\epsilon}$ such that for any $x,y\in [e^{-M}, e^{M}]$,
\begin{equation}
\begin{split}
\left| x-y \right|\le \tilde{\epsilon} &\Rightarrow \left| \log(x)- \log(y) \right| \le \epsilon\\
\Leftrightarrow \left| x- y \right| > \tilde{\epsilon} & \Leftarrow \left| \log(x)- \log(y) \right| > \epsilon.
\end{split}
\end{equation}
 We know that there is a further sub-sequence $n_{l_{m}}$ such that $(Y_{n_{l_{m}}}, W_{n_{l_{m}},1}, \ldots, W_{n_{l_{m}},k})$ converges jointly in distribution to 
\begin{equation}
(Y_{n_{l_{m}}}, W_{n_{l_{m}},1}, \ldots, W_{n_{l_{m}},k}) \stackrel{d}{\to} (H_{1},H_{2},\ldots, H_{k+1})\in (\Omega\{ n_{l_{m}} \}, \mathcal{F}\{ n_{l_{m}} \},\mathbb{P}\{ n_{l_{m}} \}).
\end{equation}
Let $\mathcal{F}\{ n_{l_{m}},1 \}\subset \mathcal{F}\{ n_{l_{m}} \}$ be the sigma algebra generated by $(H_{2},\ldots, H_{k+1})$. Here $H_{1} \stackrel{d}{=} L$ and $\left(H_{2},\ldots, H_{k+1}\right)\stackrel{d}{=}\left( Z_{1},\ldots, Z_{k} \right)$. Using the arguments same as the previous proof we see that 
\begin{equation}
\E\left[ H_{1} \left| \mathcal{F}_{n_{l_{m}},1} \right. \right]=\exp \left\{ \sum_{i=1}^{k} \frac{2\mu_{i}H_{i+1}- \mu_{i}^2 }{2 \sigma_{i}^2} \right\}.
\end{equation}
As a consequence, we have 
\begin{equation}
0\le \E\left( H_{1} - \exp \left\{ \sum_{i=1}^{k} \frac{2\mu_{i}H_{i+1}- \mu_{i}^2 }{2 \sigma_{i}^2} \right\} \right)^2\le  \exp\left\{ \sum_{i=1}^{\infty} \frac{\mu_{i}^2}{\sigma_{i}^2} \right\} - \exp\left\{ \sum_{i=1}^{k}\frac{\mu_{i}^2}{\sigma_{i}^2}\right\}.
\end{equation}
We shall choose this $k$ large enough so that 
\begin{equation}
\exp\left\{ \sum_{i=1}^{\infty} \frac{\mu_{i}^2}{\sigma_{i}^2} \right\} - \exp\left\{ \sum_{i=1}^{k}\frac{\mu_{i}^2}{\sigma_{i}^2}\right\} < \frac{\delta \tilde{\epsilon}^2}{100}.
\end{equation}
Now by Chebyshev's inequality 
\begin{equation}
\mathbb{P}\left[\left| H_{1} - \exp \left\{ \sum_{i=1}^{k} \frac{2\mu_{i}H_{i+1}- \mu_{i}^2 }{2 \sigma_{i}^2} \right\} \right|\ge \frac{\tilde{\epsilon}}{2}\right]\le \frac{\delta \tilde{\epsilon}^2}{25 \tilde{\epsilon}^2}=\frac{\delta}{25}.
\end{equation}
Since 
\begin{equation}
\left( Y_{n_{l_{m}}}, W_{n_{l_{m}},1},\ldots, W_{n_{l_{m}},k} \right) \stackrel{d}{\to} \left( H_{1},H_{2},\ldots, H_{k+1} \right)
\end{equation}
by continuous mapping theorem for in distributional convergence, we have 
\begin{equation}
Y_{n_{l_{m}}} - \exp\left\{ \sum_{i=1}^{k} \frac{2\mu_{i} W_{n_{l_{m}},i}- \mu_{i}^2}{2\sigma_{i}^2} \right\} \stackrel{d}{\to} H_{1} - \exp \left\{ \sum_{i=1}^{k} \frac{2\mu_{i}H_{i+1}- \mu_{i}^2 }{2 \sigma_{i}^2} \right\}.
\end{equation}
Since the set $[\frac{\tilde{\epsilon}}{2},\infty)$ is closed, we have by Portmanteau theorem, 
\begin{equation}
\begin{split}
&\limsup_{n_{l_{m}}} \mathbb{P}_{n_{l_{m}}}\left[ \left| Y_{n_{l_{m}}} - \exp\left\{ \sum_{i=1}^{k} \frac{2\mu_{i} W_{n_{l_{m}},i}- \mu_{i}^2}{2\sigma_{i}^2} \right\} \right| > \tilde{\epsilon}\right]\\
 &\le  \limsup_{n_{l_{m}}}\mathbb{P}_{n_{l_{m}}}\left[ \left| Y_{n_{l_{m}}} - \exp\left\{ \sum_{i=1}^{k} \frac{2\mu_{i} W_{n_{l_{m}},i}- \mu_{i}^2}{2\sigma_{i}^2} \right\} \right|\ge \frac{\tilde{\epsilon}}{2}\right]\\
& \le \frac{\delta}{25}.
\end{split}
\end{equation}
As a consequence, 
\begin{equation}
\begin{split}
&\frac{\delta}{25}
 \ge \mathbb{P}_{n_{l_{m}}}\left[ \left| Y_{n_{l_{m}}} - \exp\left\{ \sum_{i=1}^{k} \frac{2\mu_{i} W_{n_{l_{m}},i}- \mu_{i}^2}{2\sigma_{i}^2} \right\} \right|> \tilde{\epsilon}\right]\\
&\ge \mathbb{P}_{n_{l_{m}}}\left[  Y_{n_{l_{m}}} \in [e^{-M}, e^{M}] \cap  \exp \left\{ \sum_{i=1}^{k} \frac{2\mu_{i}H_{i+1}- \mu_{i}^2 }{2 \sigma_{i}^2} \right\} \in [e^{-M}, e^{M}] \cap \left|  Y_{n_{l_{m}}} - \exp \left\{ \sum_{i=1}^{k} \frac{2\mu_{i}H_{i+1}- \mu_{i}^2 }{2 \sigma_{i}^2} \right\} \right| > \tilde{\epsilon}  \right]\\
& \ge \mathbb{P}_{n_{l_{m}}}\left[ Y_{n_{l_{m}}} \in [e^{-M}, e^{M}] \cap  \exp \left\{ \sum_{i=1}^{k} \frac{2\mu_{i}H_{i+1}- \mu_{i}^2 }{2 \sigma_{i}^2} \right\} \in [e^{-M}, e^{M}] \cap \left|  \log\left(Y_{n_{l_{m}}}\right) -  \left\{ \sum_{i=1}^{k} \frac{2\mu_{i}H_{i+1}- \mu_{i}^2 }{2 \sigma_{i}^2} \right\} \right|> \epsilon \right]\\
& \ge 1-  \mathbb{P}_{n_{l_{m}}}\left[\left( Y_{n_{l_{m}}} \in [e^{-M}, e^{M}] \cap  \exp \left\{ \sum_{i=1}^{k} \frac{2\mu_{i}H_{i+1}- \mu_{i}^2 }{2 \sigma_{i}^2} \right\} \in [e^{-M}, e^{M}] \right)^{c} \right]\\
&~~~~~~~~~~ -\mathbb{P}_{n_{l_{m}}}\left[\left|  \log\left(Y_{n_{l_{m}}}\right) -  \left\{ \sum_{i=1}^{k} \frac{2\mu_{i}H_{i+1}- \mu_{i}^2 }{2 \sigma_{i}^2} \right\} \right|\le  \epsilon \right]\\
& \ge \mathbb{P}_{n_{l_{m}}}\left[\left|  \log\left(Y_{n_{l_{m}}}\right) -  \left\{ \sum_{i=1}^{k} \frac{2\mu_{i}H_{i+1}- \mu_{i}^2 }{2 \sigma_{i}^2} \right\} \right|>  \epsilon \right]- \frac{\delta}{100}\\
&\Rightarrow  \mathbb{P}_{n_{l_{m}}}\left[\left|  \log\left(Y_{n_{l_{m}}}\right) -  \left\{ \sum_{i=1}^{k} \frac{2\mu_{i}H_{i+1}- \mu_{i}^2 }{2 \sigma_{i}^2} \right\} \right|>  \epsilon \right] \le \frac{\delta}{25} + \frac{\delta}{100}< \delta.
\end{split}
\end{equation}
\hfill{$\square$}
\section{Construction of $\mathbb{P}_{n}$ and $\mathbb{Q}_{n}$ and asymptotic distribution of signed cycles}
\subsection{Construction of the measure $\mathbb{Q}_{n}$}\label{subsec:meas}
We at first give the construction of measures $\mathbb{P}_{n}$ and $\mathbb{Q}_{n}$. We assume that the random variables $\left( A_{i,j} \right)_{1\le i <j \le n}$ are defined on $\left(\Omega_{n},\mathcal{F}_{n}\right)$. 

\noindent 
In this paper $\mathbb{P}_{n}$ is simply taken to be the measure induced by $\left( A_{i,j} \right)_{1\le i <j \le n}$. We now define the measure $\mathbb{Q}_{n}$ in the following way: 
At first for any given $\sigma \in \{ -1, +1 \}^n$, we define the measure $\mathbb{Q}_{n,\sigma}$ by 
\begin{equation}\label{def:qsigma}
\frac{d\mathbb{Q}_{n,\sigma}}{d\mathbb{P}_{n}} := \exp \left\{ \sum_{i<j} \left(\frac{2\beta}{\sqrt{n}}\sigma_{i}\sigma_{j} A_{i,j}- \frac{2\beta^2}{n}\right) + \frac{\beta J}{n} \left( \sum_{i=1}^{n}\sigma_{i} \right)^2  \right\}.
\end{equation}
Observe that $\mathbb{Q}_{n,\sigma}$ is not in general a probability measure. In particular,
\begin{equation}
\int_{\Omega_{n}} d \mathbb{Q}_{n,\sigma} = \exp\left\{ \frac{\beta J}{n} \left( \sum_{i=1}^{n}\sigma_{i} \right)^2 \right\}.  
\end{equation}
Here we have used the fact that $\E\left[ \exp{tX} \right]=\exp\left\{\frac{t^2}{2}\right\}$ whenever $X \sim N(0,1)$.
Here $\Omega_{n}$ is the sample space on which the random variables $A_{i,j}$'s are defined.
Finally, we define 
\begin{equation}\label{def:q}
\mathbb{Q}_{n} := \frac{1}{\E_{\Psi_{n}}\left[  \exp\left\{ \frac{\beta J}{n}\left( \sum_{i=1}^{n}\sigma_{i} \right)^2\right\}  \right]}\sum_{\sigma \in \{-1, +1  \}^n} \frac{1}{2^n} \mathbb{Q}_{n,\sigma}.
\end{equation}
Observe that $\mathbb{Q}_{n}$ is a valid probability measure on $\Omega_{n}$.  Let
\begin{equation}
\tau_{n}:=\E_{\Psi_{n}}\left[  \exp\left\{ \frac{\beta J}{n}\left( \sum_{i=1}^{n}\sigma_{i} \right)^2\right\}  \right].
\end{equation}
One might observe that $\tau_{n}$ is the partition function of Curie-Weiss model.
From Hoeffding's inequality we have 
\begin{equation}
\begin{split}
&\mathbb{P}\left[ \frac{1}{\sqrt{n}}\left| \sum_{i=1}^{n} \sigma_{i} \right|>t \right]\le 2 \exp\left\{ - \frac{t^2}{2} \right\}\\
& \Rightarrow \mathbb{P} \left[ \exp\left\{\frac{\beta J}{n}\left( \sum_{i=1}^{n}\sigma_{i} \right)^2  \right\} \ge t \right] \le \exp\left\{ - \left( \frac{\log t}{2\beta J} \right) \right\}= \left( \frac{1}{t} \right)^{\alpha_{0}}
\end{split}
\end{equation}
for some $\alpha_{0}>1$. This makes the random variable $\exp\left\{\frac{\beta J}{n}\left( \sum_{i=1}^{n}\sigma_{i} \right)^2  \right\}$ uniformly integrable.
Hence 
\begin{equation}
\tau_{n}\to \frac{1}{\sqrt{1-2\beta J}}.
\end{equation}
 Plugging in the definition of partition function in \eqref{eq:partition}, 
it is worth noting that:
\begin{equation}
\begin{split}
\frac{d\mathbb{Q}_{n}}{d\mathbb{P}_{n}}&= \frac{1}{\tau_{n}} \sum_{\sigma \in \{-1, +1  \}^n} \frac{1}{2^n}\exp \left\{ \sum_{i<j} \left(\frac{2\beta}{\sqrt{n}}\sigma_{i}\sigma_{j} A_{i,j}- \frac{2\beta^2}{n}\right) + \frac{\beta J}{n} \left( \sum_{i=1}^{n}\sigma_{i} \right)^2  \right\}\\
&= \frac{1}{\tau_{n}} \exp\left\{-(n-1)\beta^2 + \beta J\right\} \exp\left\{ - \frac{\beta}{\sqrt{n}} \sum_{i=1}^{n} A_{i,i}  -\beta J'\right\}Z_{n}(\beta).
\end{split}
\end{equation}
So in order to prove Theorem \ref{thm:asymptotic} it is enough to prove a central limit theorem for $\log\left(\frac{d\mathbb{Q}_{n}}{d\mathbb{P}_{n}}\right)$ and to prove that $\log\left(\frac{d\mathbb{Q}_{n}}{d\mathbb{P}_{n}}\right)$ is asymptotically independent of $\frac{1}{\sqrt{n}}\sum_{i=1}^{n} A_{i,i}$.
\subsection{Asymptotic distribution of $C_{n,i}$'s under $\mathbb{P}_{n}$ and $\mathbb{Q}_{n}$}
In order to derive the limiting distribution of $C_{n,i}$'s under $\mathbb{Q}_{n}$ we at first need to define another sequence of measure $\mathbb{Q}_{n}'$. We shall at first derive the limiting distribution of $C_{n,i}$'s under $\mathbb{Q}_{n}'$ and then we shall find the limiting distribution of $C_{n,i}$'s under $\mathbb{Q}_{n}$. \\
Let for any given $\sigma \in \{-1,+1 \}^n$, $\mathbb{Q}_{n,\sigma}'$ be defined as 
\begin{equation}
\frac{d\mathbb{Q}_{n,\sigma}'}{d\mathbb{P}_{n}}= \exp \left\{ \sum_{i<j} \left(\frac{2\beta}{\sqrt{n}}\sigma_{i}\sigma_{j} A_{i,j}- \frac{2\beta^2}{n}\right) \right\}.
\end{equation}
Observe that $\mathbb{Q}_{n,\sigma}'$ is a probability measure. In fact $\left(A_{i,j}\right)_{1 \le i < j \le n }\left|_{\mathbb{Q}_{n,\sigma}'}  \right.$ are independent normal random variables with $A_{i,j}\left|_{\mathbb{Q}_{n,\sigma}'}  \right. \sim N\left( \frac{2\beta}{\sqrt{n}}\sigma_{i}\sigma_{j},1 \right)$. Here $\left(A_{i,j}\right)_{1 \le i < j \le n }\left|_{\mathbb{Q}_{n,\sigma}'}  \right.$  denote the joint distribution of the random variables $\left(A_{i,j}\right)_{1\le i,j \le n}$ under the measure $\mathbb{Q}_{n,\sigma}'$. Finally
\begin{equation}
\mathbb{Q}_{n}':= \frac{1}{2^{n}}\sum_{\sigma\in \{-1, 1 \}^{n}}  \mathbb{Q}_{n,\sigma}'.
\end{equation}

\noindent 
The first result in this section gives the asymptotic distribution of $C_{n,i}$'s under $\mathbb{P}_{n}$ and $\mathbb{Q}_{n}$. 
\begin{proposition}\label{prop:signdistr}
\begin{enumerate}
\item Under $\mathbb{P}_{n}$, we have for any $2 \le k_{1}< k_{2}\ldots < k_{l} = o\left( \sqrt{\log(n)} \right)$ with $l$ fixed,  
\begin{equation}
\left( \frac{C_{n,k_{1}}- (n-1)\mathbb{I}_{k_{1}=2}}{\sqrt{2k_{1}}}, \ldots, \frac{C_{n,k_{l}}}{\sqrt{2k_{l}}} \right) \stackrel{d}{\to} N_{l}(0, I_{l}).
\end{equation}
\item Let $\Psi_{n}$ be the uniform probability measure on the hyper cube $\{ -1, +1\}^{n}$. Then there exists a set $S_{n}$ with $\Psi_{n}\left(S_{n}\right) \to 0$, we have for all $\sigma \in S_{n}^{c}$, under $\mathbb{Q}_{n,\sigma}'$
\begin{equation}\label{res:qnsig}
\left( \frac{C_{n,k_{1}}- (n-1)\mathbb{I}_{k_{1}=2}- \mu_{k_{1}}}{\sqrt{2k_{1}}}, \ldots, \frac{C_{n,k_{l}}-\mu_{k_{l}}}{\sqrt{2k_{l}}} \right) \stackrel{d}{\to} N_{l}(0, I_{l})
\end{equation}
where $\mu_{i}:= \left( 2 \beta \right)^{i}$. This implies under $\mathbb{Q}_{n}'$,
\begin{equation}
\left( \frac{C_{n,k_{1}}- (n-1)\mathbb{I}_{k_{1}=2}- \mu_{k_{1}}}{\sqrt{2k_{1}}}, \ldots, \frac{C_{n,k_{l}}-\mu_{k_{l}}}{\sqrt{2k_{l}}} \right) \stackrel{d}{\to} N_{l}(0, I_{l}).
\end{equation}
\item Finally, $C_{n,1}\stackrel{d}{\to} N(0,1)$ under $\mathbb{P}_{n}$ and is asymptotically independent of the process $\{ C_{n,k}- (n-1)\mathbb{I}_{k=2} \}_{k\ge 2}$.
\end{enumerate}
\end{proposition}
Here $N_{l}(\mu,\Sigma)$ denotes an $l$ dimensional normal random vector with mean parameter $\mu$ and variance parameter $\Sigma$.
The proof of Proposition \ref{prop:signdistr} is given in Section \ref{sec:propsign}.
With Proposition \ref{prop:signdistr}, we now give the asymptotic distribution of $C_{n,i}$'s under $\mathbb{Q}_{n}$.
\begin{proposition}\label{prop:signdistrnew}
Under $\mathbb{Q}_{n}$, we have for any $2 \le k_{1}< k_{2}\ldots < k_{l} = o\left( \sqrt{\log(n)} \right)$ with $l$ fixed,  
\begin{equation}
\left( \frac{C_{n,k_{1}}- (n-1)\mathbb{I}_{k_{1}=2}-\mu_{k_1}}{\sqrt{2k_{1}}}, \ldots, \frac{C_{n,k_{l}}-\mu_{k_{l}}}{\sqrt{2k_{l}}} \right) \stackrel{d}{\to} N_{l}(0, I_{l}).
\end{equation}
\end{proposition}
\begin{proof}
We assume Proposition \ref{prop:signdistr} and give the proof. We need to prove for any bounded continuous function $f: \mathbb{R}^{l} \to \mathbb{R}$,
\begin{equation}
\int f\left(\frac{C_{n,k_{1}}- (n-1)\mathbb{I}_{k_{1}=2}-\mu_{k_1}}{\sqrt{2k_{1}}}, \ldots, \frac{C_{n,k_{l}}-\mu_{k_{l}}}{\sqrt{2k_{l}}}\right) d \mathbb{Q}_{n} \to \E \left[ f(Z_{k_{1}},\ldots, Z_{k_{l}}) \right]
\end{equation}
where $Z_{k_{1}},\ldots, Z_{k_{l}}$ are independent standard Gaussian random variables.
Now 
\begin{equation}
\begin{split}
&\int_{\Omega_{n}} f\left(\frac{C_{n,k_{1}}- (n-1)\mathbb{I}_{k_{1}=2}-\mu_{k_1}}{\sqrt{2k_{1}}}, \ldots, \frac{C_{n,k_{l}}-\mu_{k_{l}}}{\sqrt{2k_{l}}}\right) d \mathbb{Q}_{n}\\
 & = \frac{1}{2^n}\sum_{\sigma \in \{ -1,+1 \}^{n}} \int_{\Omega_{n}} f\left(\frac{C_{n,k_{1}}- (n-1)\mathbb{I}_{k_{1}=2}-\mu_{k_1}}{\sqrt{2k_{1}}}, \ldots, \frac{C_{n,k_{l}}-\mu_{k_{l}}}{\sqrt{2k_{l}}}\right) d \mathbb{Q}_{n,\sigma}\\
 &= \frac{1}{2^n}\sum_{\sigma \in \{ -1,+1 \}^{n}}\int_{\Omega_{n}} f\left(\frac{C_{n,k_{1}}- (n-1)\mathbb{I}_{k_{1}=2}-\mu_{k_1}}{\sqrt{2k_{1}}}, \ldots, \frac{C_{n,k_{l}}-\mu_{k_{l}}}{\sqrt{2k_{l}}}\right) \frac{d\mathbb{Q}_{n,\sigma}}{d\mathbb{P}_{n}} d\mathbb{P}_{n}\\
 &=   \frac{1}{\tau_{n}}\frac{1}{2^n}\sum_{\sigma \in \{ -1,+1 \}^{n}}\int_{\Omega_{n}} f\left(\frac{C_{n,k_{1}}- (n-1)\mathbb{I}_{k_{1}=2}-\mu_{k_1}}{\sqrt{2k_{1}}}, \ldots, \frac{C_{n,k_{l}}-\mu_{k_{l}}}{\sqrt{2k_{l}}}\right)\exp\left\{ \frac{\beta J}{n}\left(\sum \sigma_{i}\right)^2\right\} \frac{d\mathbb{Q}_{n,\sigma}'}{d\mathbb{P}_{n}} d\mathbb{P}_{n}\\
 &= \frac{1}{\tau_{n}} \frac{1}{2^n}\sum_{\sigma \in \{ -1,+1 \}^{n}} \exp\left\{ \frac{\beta J}{n}\left(\sum \sigma_{i}\right)^2\right\} F_{n}(\sigma) \\
 &= \frac{1}{\tau_{n}} \E_{\Psi_{n}} \left[ \exp\left\{ \frac{\beta J}{n}\left(\sum \sigma_{i}\right)^2\right\} F_{n}(\sigma) \right]
\end{split}
\end{equation}
Here $F_{n}(\sigma)=  \int_{\Omega_{n}} f\left(\frac{C_{n,k_{1}}- (n-1)\mathbb{I}_{k_{1}=2}-\mu_{k_1}}{\sqrt{2k_{1}}}, \ldots, \frac{C_{n,k_{l}}-\mu_{k_{l}}}{\sqrt{2k_{l}}}\right) \frac{d\mathbb{Q}_{n,\sigma}'}{d\mathbb{P}_{n}} d\mathbb{P}_{n}.$ From \eqref{res:qnsig} of Proposition \ref{prop:signdistr}, we know that under the measure $\Psi_{n}(\cdot)$, $F_{n}(\sigma) \stackrel{p}{\to} \E \left[ f(Z_{k_{1}},\ldots, Z_{k_{l}}) \right]$.
Now from central limit theorem, 
\begin{equation}
\frac{1}{n}\left(\sum \sigma_{i}\right)^2 \stackrel{d}{\to} Y 
\end{equation}
where $Y$ is a Chi-squared random variable with $1$ degree of freedom.
So by Slutsky's theorem we have under the measure $\Psi_{n}$
\begin{equation}
F_{n}(\sigma)\exp\left\{ \frac{\beta J}{n}\left(\sum \sigma_{i}\right)^2\right\} \stackrel{d}{\to} \E \left[ f(Z_{k_{1}},\ldots, Z_{k_{l}}) \right]\exp\left\{ \beta J Y\right\}.
\end{equation}
 Further, from Hoeffding's inequality we also have when $\beta J <\frac{1}{2}$, the sequence $\exp\left\{ \frac{\beta J}{n}\left(\sum \sigma_{i}\right)^2\right\}$ is uniformly integrable. Since the random variables $F_{n}(\sigma)$'s are uniformly bounded, the sequence $ F_{n}(\sigma)\exp\left\{ \frac{\beta J}{n}\left(\sum \sigma_{i}\right)^2\right\} $ is also uniformly integrable. As a consequence, 
 \begin{equation}
 \E_{\Psi_{n}} \left[ \exp\left\{ \frac{\beta J}{n}\left(\sum \sigma_{i}\right)^2\right\} F_{n}(\sigma) \right] \to \E \left[ f(Z_{k_{1}},\ldots, Z_{k_{l}}) \right] \frac{1}{\sqrt{1- 2\beta J}}.
 \end{equation}  
   
\end{proof}
\begin{remark}
Along with Proposition \ref{prop:norcont}, Proposition \ref{prop:signdistrnew} is another important Result to prove Theorem \ref{thm:asymptotic}. In particular, Proposition \ref{prop:signdistrnew} allows us to verify condition $(iii)$ of Proposition \ref{prop:norcont} in the context of Theorem \ref{thm:asymptotic}.
One might observe that in proof of Proposition \ref{prop:signdistrnew} there are two important facts. First of all, part (2) of Proposition \ref{prop:signdistr} where one proves the asymptotic normality of the signed cycles holds with same parameter for almost all $\sigma$'s. This makes $F_{n}(\sigma)$ converge to $\E \left[ f(Z_{k_{1}},\ldots, Z_{k_{l}}) \right]$ for almost all $\sigma$'s. Secondly, it is also important that the partition function of the Curie-Weiss model in high temperature has a limit $\frac{1}{\sqrt{1-2\beta J}}$. Hence $\frac{1}{\tau_{n}}$ cancels out the $\frac{1}{\sqrt{1-2\beta J}}$ factor giving us the needed result.
\end{remark}
\section{Proof of Theorem \ref{thm:asymptotic}}
As mentioned in subsection \ref{subsec:meas}, we at first prove a central limit theorem for $\log\left(  \frac{d\mathbb{Q}_{n}}{d\mathbb{P}_{n}}\right)\left| \mathbb{P}_{n} \right.$ and finally proving $\log\left( \frac{d\mathbb{Q}_{n}}{d\mathbb{P}_{n}} \right)\left| \mathbb{P}_{n} \right.$ is asymptotically independent of $C_{n,1}$. The main idea is to use Proposition \ref{prop:norcont} to a class of measure $\tilde{\mathbb{Q}}_{n}$ which is close to $\mathbb{Q}_{n}$ in total variation distance. We now give a formal proof of Theorem \ref{thm:asymptotic}.\\
\textbf{Proof of Theorem \ref{thm:asymptotic}:}\\
We at first prove the central limit theorem for $\log\left(\frac{d\mathbb{Q}_{n}}{d\mathbb{P}_{n}}\right)\left| \mathbb{P}_{n} \right.$. The proof is broken into two steps as follows.

\noindent 
\textbf{Step 1 (Construction of the measure $\tilde{\mathbb{Q}}_{n}$) :}
To begin with we shall consider a set $\Omega(\sigma)_{n} \subset \{ -1,+1 \}^{n}$ such that $\Psi_{n}\left(\Omega(\sigma)_{n}\right)\to 1$. The precise definition of $\Omega(\sigma)_{n}$ will be provided later. Now we consider the measure $\tilde{\mathbb{Q}}_{n}$ as follows
\begin{equation}
\tilde{\mathbb{Q}}_{n}= \frac{1}{\E_{\Psi_{n}}\left[ \mathbb{I}_{\Omega(\sigma)_{n}} \exp\left\{ \frac{\beta J}{n}\left( \sum_{i=1}^{n}\sigma_{i} \right)^2\right\}  \right]}\sum_{\sigma \in \Omega(\sigma)_{n}} \frac{1}{2^n}\mathbb{Q}_{n,\sigma}= \frac{1}{\tilde{\tau}_{n}}\sum_{\sigma \in \Omega(\sigma)_{n}} \frac{1}{2^n}\mathbb{Q}_{n,\sigma}
\end{equation} 
where we define 
\begin{equation}
\tilde{\tau}_{n}:= \E_{\Psi_{n}}\left[ \mathbb{I}_{\Omega(\sigma)_{n}} \exp\left\{ \frac{\beta J}{n}\left( \sum_{i=1}^{n}\sigma_{i} \right)^2\right\}  \right].
\end{equation}
Since the sequence of random variables $\exp\left\{ \frac{\beta J}{n}\left( \sum_{i=1}^{n}\sigma_{i} \right)^2\right\}$ is uniformly integrable it follows that for any sequence of sets $\Omega_{n}(\sigma)$ such that $\Psi_{n}\left[ \Omega_{n}(\sigma) \right] \to 1,$
\begin{equation}
\tilde{\tau}_{n} \to \frac{1}{\sqrt{1-2\beta J}}.
\end{equation}
Now we prove the sequences of measures $\mathbb{Q}_{n}$ and $\tilde{\mathbb{Q}}_{n}$ are close in the total variation sense. Let $A_{n}\in \mathcal{F}_{n}$ be a sequence of measurable sets. We have 
\begin{equation}\label{eq:tvcal}
\begin{split}
&\left|\mathbb{Q}_{n}(A_{n})- \tilde{\mathbb{Q}}_{n}(A_{n})  \right|\\
&=\left| \frac{1}{\tau_{n}}\sum_{\sigma \in \{ -1,+1 \}^n} \frac{1}{2^n} \int_{A_{n}}\frac{d\mathbb{Q}_{n,\sigma}}{d\mathbb{P}_{n}}d\mathbb{P}_{n}- \frac{1}{\tilde{\tau}_{n}} \sum_{\sigma \in \Omega_{n}(\sigma) }\frac{1}{2^n} \int_{A_{n}}\frac{d\mathbb{Q}_{n,\sigma}}{d\mathbb{P}_{n}}d\mathbb{P}_{n}\right|\\
&\le \left| \frac{1}{\tau_{n}} \sum_{\sigma \in \Omega_{n}(\sigma)^{c}} \frac{1}{2^n} \int_{A_{n}}\frac{d\mathbb{Q}_{n,\sigma}}{d\mathbb{P}_{n}}d\mathbb{P}_{n}  \right| + \left|\left( \frac{1}{\tau_{n}} -\frac{1}{\tilde{\tau}_{n}}\right)\sum_{\sigma \in \Omega_{n}(\sigma)}\frac{1}{2^n} \int_{A_{n}} \frac{d\mathbb{Q}_{n,\sigma}}{d\mathbb{P}_{n}}d\mathbb{P}_{n}\right|\\
&\le  \left| \frac{1}{\tau_{n}} \sum_{\sigma \in \Omega_{n}(\sigma)^{c}} \frac{1}{2^n} \int_{\Omega_{n}}\frac{d\mathbb{Q}_{n,\sigma}}{d\mathbb{P}_{n}}d\mathbb{P}_{n}  \right| + \left| \left( \frac{1}{\tau_{n}} -\frac{1}{\tilde{\tau}_{n}}\right) \right|\left| \sum_{\sigma \in \Omega_{n}(\sigma)}\frac{1}{2^n} \int_{\Omega_{n}} \frac{d\mathbb{Q}_{n,\sigma}}{d\mathbb{P}_{n}}d\mathbb{P}_{n} \right|\\
& \le \left|\frac{1}{\tau_{n}} \E_{\Psi_{n}}\left[ \mathbb{I}_{\Omega(\sigma)_{n}^{c}} \exp\left\{ \frac{\beta J}{n}\left( \sum_{i=1}^{n}\sigma_{i} \right)^2\right\} \right] \right|+ \left| \left( \frac{1}{\tau_{n}} -\frac{1}{\tilde{\tau}_{n}}\right) \right|\E_{\Psi_{n}}\left[ \mathbb{I}_{\Omega(\sigma)_{n}} \exp\left\{ \frac{\beta J}{n}\left( \sum_{i=1}^{n}\sigma_{i} \right)^2\right\} \right]
\end{split}
\end{equation}
Observe that the final expression in \eqref{eq:tvcal} does not depend on the set $A_{n}$ and also it has been argued earlier that the final expression in \eqref{eq:tvcal} converges to $0$.  
By Proposition \ref{prop:signdistrnew}, under the measure $\tilde{\mathbb{Q}}_{n}$ the random variables 
for any $2 \le k_{1}< k_{2}\ldots < k_{l} = o\left( \sqrt{\log(n)} \right)$ with $l$ fixed,  
\begin{equation}
\left( \frac{C_{n,k_{1}}- (n-1)\mathbb{I}_{k_{1}=2}-\mu_{k_1}}{\sqrt{2k_{1}}}, \ldots, \frac{C_{n,k_{l}}-\mu_{k_{l}}}{\sqrt{2k_{l}}} \right) \stackrel{d}{\to} N_{l}(0, I_{l}).
\end{equation}
Now we prove that $\limsup_{n \to \infty}\E_{\mathbb{P}_{n}}\left[\left(\frac{d\tilde{\mathbb{Q}}_{n}}{d\mathbb{P}_{n}}\right)^2\right]\le \exp\left\{ \sum_{k=2}^{\infty} \frac{\mu_{k}^2}{\sigma_{k}^2} \right\}$ where $\mu_{k}= (2\beta)^{k}$. This will allow us to use Proposition \ref{prop:norcont} for $\tilde{\mathbb{Q}}_{n}$. In particular, we shall get $\left( \frac{d\tilde{\mathbb{Q}}_{n}}{d\mathbb{P}_{n}} \right)\left|\mathbb{P}_{n}\right.$ has a normal limiting distribution. Once this is done, the limiting distribution of $\frac{d\mathbb{Q}_{n}}{d\mathbb{P}_{n}}\left| \mathbb{P}_{n} \right.$ can be derived by the following arguments which proves $$\frac{d\mathbb{Q}_{n}}{d\mathbb{P}_{n}}- \frac{d\tilde{\mathbb{Q}}_{n}}{d\mathbb{P}_{n}}\left|\mathbb{P}_{n} \right. \stackrel{p}{\to}  0.$$

\noindent 
Since both $\tau_{n}$ and $\tilde{\tau}_{n}$ have the same finite limit, the random variable  $$\tilde{Y}_{n}:=\frac{\tilde{\tau_{n}}}{\tau_{n}}\frac{d\tilde{\mathbb{Q}}_{n}}{d\mathbb{P}_{n}}\left|\mathbb{P}_{n}  \right.$$ has the same limiting distribution as $\frac{d\tilde{\mathbb{Q}}_{n}}{d\mathbb{P}_{n}}\left| \mathbb{P}_{n} \right.$. In particular, 
\begin{equation}
\left(\tilde{Y}_{n}- \frac{d\tilde{\mathbb{Q}}_{n}}{d\mathbb{P}_{n}}\right)\left| \mathbb{P}_{n} \right. \stackrel{p}{\to} 0.
\end{equation}
So it is enough to prove $$\left(\frac{d\mathbb{Q}_{n}}{d\mathbb{P}_{n}}- \tilde{Y}_{n} \right)\left|\mathbb{P}_{n}\right. \stackrel{p}{\to} 0.$$ However, 
\begin{equation}
\begin{split}
0\le \frac{d\mathbb{Q}_{n}}{d\mathbb{P}_{n}}- \tilde{Y}_{n} &= \frac{1}{\tau_{n}}\left(\sum_{\sigma \in \Omega_{n}(\sigma)^{c}}\frac{1}{2^n} \frac{d\mathbb{Q}_{n,\sigma}}{d\mathbb{P}_{n}}\right)\\
\Rightarrow  \E_{\mathbb{P}_{n}}\left[\frac{d\mathbb{Q}_{n}}{d\mathbb{P}_{n}}- \tilde{Y}_{n}  \right] &= \frac{1}{\tau_{n}}\E_{\Psi_{n}}\left[ \mathbb{I}_{\Omega_{n}(\sigma)^{c}}\exp\left\{ \frac{\beta J}{n}\left( \sum_{i=1}^{n}\sigma_{i} \right)^2\right\}  \right]\to 0.
\end{split}
\end{equation}
This completes the proof of 
\begin{equation}
\left(\frac{d\mathbb{Q}_{n}}{d\mathbb{P}_{n}}- \frac{d\tilde{\mathbb{Q}}_{n}}{d\mathbb{P}_{n}}\right)\left| \mathbb{P}_{n} \right. \stackrel{p}{\to} 0.
\end{equation}
\textbf{Step 2 \bigg(Upper bounding $\E_{\mathbb{P}_{n}}\left[ \left( \frac{d\tilde{\mathbb{Q}}_{n}}{d\mathbb{P}_{n}} \right)^2 \right]$ \bigg ):}

\noindent 
We know that 
\begin{equation}\label{eq:loglikemombound}
\begin{split}
&\left( \frac{d\tilde{\mathbb{Q}}_{n}}{d\mathbb{P}_{n}} \right)^2 = \left(\frac{1}{\tilde{\tau_{n}}}\right)^2\frac{1}{4^n}\sum_{\sigma \in \Omega(\sigma)_{n}} \sum_{\sigma' \in \Omega(\sigma)_{n}}\frac{d\mathbb{Q}_{n,\sigma}}{d\mathbb{P}_{n}}\frac{d\mathbb{Q}_{n,\sigma'}}{d\mathbb{P}_{n}}\\
&= \left(\frac{1}{\tilde{\tau_{n}}}\right)^2\frac{1}{4^n}\sum_{\sigma \in \Omega(\sigma)_{n}} \sum_{\sigma' \in \Omega(\sigma)_{n}}\exp\left\{ \sum_{i<j} \left(\frac{2\beta}{\sqrt{n}}A_{i,j}\left( \sigma_{i}\sigma_{j} + \sigma_{i}' \sigma_{j}'\right)- \frac{4\beta^2}{n}\right) + \frac{\beta J}{n}\left( \sum_{i=1}^{n}\sigma_{i} \right)^2 + \frac{\beta J}{n}\left( \sum_{i=1}^{n}\sigma'_{i} \right)^2 \right\}\\
& \Rightarrow \E_{\mathbb{P}_{n}}\left[ \left( \frac{d\tilde{\mathbb{Q}}_{n}}{d\mathbb{P}_{n}} \right)^2  \right]\\
&= \left(\frac{1}{\tilde{\tau_{n}}}\right)^2\frac{1}{4^n}\sum_{\sigma \in \Omega(\sigma)_{n}} \sum_{\sigma' \in \Omega(\sigma)_{n}}\exp\left\{\sum_{i<j}\left( \frac{2\beta^2}{n}\left( \sigma_{i}\sigma_{j}+ \sigma'_{i}\sigma'_{j} \right)^2- \frac{4\beta^2}{n} \right)+ \frac{\beta J}{n}\left( \sum_{i=1}^{n}\sigma_{i} \right)^2 + \frac{\beta J}{n}\left( \sum_{i=1}^{n}\sigma'_{i} \right)^2 \right\}\\
&= \left(\frac{1}{\tilde{\tau_{n}}}\right)^2\frac{1}{4^n}\sum_{\sigma \in \Omega(\sigma)_{n}} \sum_{\sigma' \in \Omega(\sigma)_{n}}\exp\left\{\sum_{i<j}\left( \frac{4\beta^2}{n} \sigma_{i}\sigma_{j} \sigma'_{i}\sigma'_{j} \right)+ \frac{\beta J}{n}\left( \sum_{i=1}^{n}\sigma_{i} \right)^2 + \frac{\beta J}{n}\left( \sum_{i=1}^{n}\sigma'_{i} \right)^2 \right\}\\
&=  \left(\frac{1}{\tilde{\tau_{n}}}\right)^2\frac{1}{4^n}\sum_{\sigma \in \Omega(\sigma)_{n}} \sum_{\sigma' \in \Omega(\sigma)_{n}}\exp\left\{ \frac{2\beta^2}{n}\left( \sum_{i=1}^{n} \sigma_{i}\sigma'_{i} \right)^2 -2\beta^2\ + \frac{\beta J}{n}\left( \sum_{i=1}^{n}\sigma_{i} \right)^2 + \frac{\beta J}{n}\left( \sum_{i=1}^{n}\sigma'_{i} \right)^2  \right\}\\
&= \exp\left\{ -2\beta^2 \right\}\left(\frac{1}{\tilde{\tau_{n}}}\right)^2 \E_{\Psi_{n}\otimes \Psi_{n}}\left[ \mathbb{I}_{\sigma \in \Omega(\sigma)_{n}} \mathbb{I}_{\sigma' \in \Omega(\sigma)_{n}}\exp\left\{ \frac{2\beta^2}{n}\left( \sum_{i=1}^{n} \sigma_{i}\sigma'_{i} \right)^2  + \frac{\beta J}{n}\left( \sum_{i=1}^{n}\sigma_{i} \right)^2 + \frac{\beta J}{n}\left( \sum_{i=1}^{n}\sigma'_{i} \right)^2 \right\} \right]
\end{split}
\end{equation}
Here $\Psi_{n}\otimes \Psi_{n}$ denote the two fold product of the uniform probability measure on $\{-1 , 1\}^{n} \times \{ -1, 1 \}^{n}$.

Observe that the random variable 
\begin{equation}\label{def:exp}
\mathbb{I}_{\sigma \in \Omega(\sigma)_{n}} \mathbb{I}_{\sigma' \in \Omega(\sigma)_{n}}\exp\left\{ \frac{2\beta^2}{n}\left( \sum_{i=1}^{n} \sigma_{i}\sigma'_{i} \right)^2  + \frac{\beta J}{n}\left( \sum_{i=1}^{n}\sigma_{i} \right)^2 + \frac{\beta J}{n}\left( \sum_{i=1}^{n}\sigma'_{i} \right)^2 \right\} \stackrel{d}{\to} \exp\left\{ 2\beta^2 Y_{1} + \beta J Y_{2} + \beta J Y_{3}  \right\}
\end{equation}
where $Y_{1},Y_{2},Y_{3}$ are three independent chi-square random variables each with one degree of freedom. Our target is to prove the random variable in the l.h.s. of \eqref{def:exp} is uniformly integrable.
\noindent 
This done by proving 
\begin{equation}
\limsup_{n \to \infty}\E_{\Psi_{n}\otimes \Psi_{n}}\left[ \mathbb{I}_{\sigma \in \Omega(\sigma)_{n}} \mathbb{I}_{\sigma' \in \Omega(\sigma)_{n}}\exp\left\{(1+\eta) \left(\frac{2\beta^2}{n}\left( \sum_{i=1}^{n} \sigma_{i}\sigma'_{i} \right)^2  + \frac{\beta J}{n}\left( \sum_{i=1}^{n}\sigma_{i} \right)^2 + \frac{\beta J}{n}\left( \sum_{i=1}^{n}\sigma'_{i} \right)^2 \right\}\right) \right]<\infty
\end{equation} 
for sufficiently small $\eta$.
We at first write
\begin{equation}
\begin{split}
&=\E_{\Psi_{n}\otimes \Psi_{n}}\left[ \mathbb{I}_{\sigma \in \Omega(\sigma)_{n}} \mathbb{I}_{\sigma' \in \Omega(\sigma)_{n}}\exp\left\{ (1+\eta)\left(\frac{2\beta^2}{n}\left( \sum_{i=1}^{n} \sigma_{i}\sigma'_{i} \right)^2  + \frac{\beta J}{n}\left( \sum_{i=1}^{n}\sigma_{i} \right)^2 + \frac{\beta J}{n}\left( \sum_{i=1}^{n}\sigma'_{i} \right)^2\right) \right\} \right]\\
&= \E\left[\E\left[ \mathbb{I}_{\sigma \in \Omega(\sigma)_{n}} \mathbb{I}_{\sigma' \in \Omega(\sigma)_{n}}\exp\left\{ (1+\eta)\left(\frac{2\beta^2}{n}\left( \sum_{i=1}^{n} \sigma_{i}\sigma'_{i} \right)^2  + \frac{\beta J}{n}\left( \sum_{i=1}^{n}\sigma_{i} \right)^2 + \frac{\beta J}{n}\left( \sum_{i=1}^{n}\sigma'_{i} \right)^2 \right)\right\} \left| \sigma \right.\right]\right]\\
&= \E\left[ \mathbb{I}_{\sigma \in \Omega(\sigma)_{n}} \exp\left\{ (1+\eta)\frac{\beta J}{n}\left( \sum_{i=1}^{n}\sigma_{i} \right)^2\right\}\E\left[ \mathbb{I}_{\sigma' \in \Omega(\sigma)_{n}}\exp\left\{(1+\eta)\left(\frac{2\beta^2}{n}\left( \sum_{i=1}^{n} \sigma_{i}\sigma'_{i} \right)^2 + \frac{\beta J}{n}\left( \sum_{i=1}^{n}\sigma'_{i} \right)^2\right)  \right\}\left| \sigma \right. \right]\right]\\
&= \E\left[ \mathbb{I}_{\sigma \in \Omega(\sigma)_{n}} \exp\left\{ (1+\eta)\frac{\beta J}{n}\left( \sum_{i=1}^{n}\sigma_{i} \right)^2 \right\}\E\left[ \mathbb{I}_{\sigma' \in \Omega(\sigma)_{n}}\exp\left\{(1+\eta)\frac{1}{n}\left(\sigma'\right)^{\mathrm{T}} A'A \left(\sigma'\right)  \right\}\left| \sigma \right. \right] \right]\\
&\le \E\left[ \mathbb{I}_{\sigma \in \Omega(\sigma)_{n}} \exp\left\{ (1+\eta)\frac{\beta J}{n}\left( \sum_{i=1}^{n}\sigma_{i} \right)^2 \right\}\E\left[ \exp\left\{(1+\eta)\frac{1}{n}\left(\sigma'\right)^{\mathrm{T}} A^{\mathrm{T}}A \left(\sigma'\right)  \right\}\left| \sigma \right. \right] \right]. 
\end{split}
\end{equation}
Here $\mathrm{T}$ denotes the transpose of a matrix  
and the matrix $A_{2\times n}$ is given by 
\begin{equation}
A=\left(
\begin{array}{llll}
\beta J& \beta J & \ldots & \beta J\\
2\beta^2\sigma_{1}& 2\beta^2\sigma_{2} & \ldots & 2\beta^2\sigma_{n}
\end{array}
\right).
\end{equation}
Since $\E\left[\exp\left\{ \alpha^{\mathrm{T}}\sigma' \right\}\right]\le \exp\left\{ \frac{1}{2} || \alpha ||^{2} \right\}$ for any $\alpha \in \mathbb{R}^{n}$, we have the following tail estimate by Theorem 1 and Remark 1 of \citet{HSU}:
\begin{equation}
\mathbb{P}\left[ \frac{1}{n}\left(\sigma'\right)^{\mathrm{T}} A^{\mathrm{T}}A \left(\sigma'\right)\ge  \mathrm{tr}(\Sigma) +2 \sqrt{\mathrm{tr}(\Sigma^2)t} + 2 || \Sigma ||t\left| \sigma \right. \right]\le e^{-t}
\end{equation}
where $\Sigma= \frac{1}{\sqrt{n}}A $. Observe that the nonzero eigenvalues of $\Sigma$ are same as the nonzero eigenvalues of $\frac{1}{n}AA^{\mathrm{T}}$. Now 
\begin{equation}  
\frac{1}{n}AA^{\mathrm{T}}=\left( 
\begin{array}{ll}
\beta J& 2\beta^3J \left( \frac{1}{n}\sum_{i=1}^{n}\sigma_{i} \right)\\
2\beta^3J\left( \frac{1}{n} \sum_{i=1}^{n}\sigma_{i} \right) & 2\beta^2
\end{array}
\right).
\end{equation}
We now choose the set $$\Omega(\sigma)_{n}:= \left\{ \frac{1}{n}\sum_{i=1}^{n}\sigma_{i} \le \delta_{n}\right\}$$ for some $\delta_{n} \to 0$ as $n \to \infty$. The existence of such $\Omega(\sigma)_{n}$ is ensured by weak law of large numbers. Now by Weyl's interlacing inequality, we have the eigenvalues of $\frac{1}{n}AA^{\mathrm{T}}$ are given by $\left\{\beta J+ O(\delta_{n}), 2\beta^2 +O(\delta_{n}) \right\} $. Also note that on $\Omega(\sigma)_{n}$, $\mathrm{tr}(\Sigma)$ and $\mathrm{tr}(\Sigma^2)$ remain uniformly bounded. So given any $\epsilon>0$ we can find a $t_{0}$ large enough such that 
\begin{equation}
\mathrm{tr}(\Sigma) +2 \sqrt{\mathrm{tr}(\Sigma^2)t} < \epsilon 2 || \Sigma ||t
\end{equation} 
for all $t>t_{0}.$ As a consequence, for all $t>t_{0}$
\begin{equation}
\begin{split}
&\mathbb{P}\left[ \frac{1}{n}\left(\sigma'\right)^{\mathrm{T}} A^{\mathrm{T}}A \left(\sigma'\right)\ge (1+\epsilon) 2|| \Sigma ||t\left| \sigma \right. \right]\\
&\le \mathbb{P}\left[ \frac{1}{n}\left(\sigma'\right)^{\mathrm{T}} A^{\mathrm{T}}A \left(\sigma'\right)\ge  \mathrm{tr}(\Sigma) +2 \sqrt{\mathrm{tr}(\Sigma^2)t} + 2 || \Sigma ||t\left| \sigma \right. \right]< e^{-t}\\
& \Rightarrow \mathbb{P}\left[ (1+\eta)\frac{1}{n}\left(\sigma'\right)^{\mathrm{T}} A^{\mathrm{T}}A \left(\sigma'\right) \ge \log(t) \left| \sigma \right.\right]\le t^\frac{-1}{2(1+\epsilon)(1+\eta)||\Sigma||} ~ \forall ~ t> \tilde{t}_{0}.
\end{split}
\end{equation}
where $\tilde{t}_{0}$ is another deterministic constant. Here the last step comes from replacing $(1+\eta)(1+\epsilon) 2|| \Sigma ||t$ by $\log(t)$. Since $\max\left\{ \beta J, 2 \beta^2 \right\}< \frac{1}{2}$, we can choose $\epsilon $ and $\eta$ small enough such that  
\begin{equation}
\frac{1}{2(1+\epsilon)(1+\eta)||\Sigma||} >\alpha_{0}>1.
\end{equation}
As a consequence, 
\begin{equation}\label{eqn:argjoint}
\begin{split}
\mathbb{I}_{\sigma \in \Omega(\sigma)_{n}}\E\left[ \exp\left\{(1+\eta)\frac{1}{n}\left(\sigma'\right)^{\mathrm{T}} A^{\mathrm{T}}A \left(\sigma'\right)  \right\}\left| \sigma \right. \right]\le \tilde{t}_{0} + \int_{t > \tilde{t}_{0}} \frac{1}{t^{\alpha_0}} dt= \tilde{t}_{0} + \frac{1}{\alpha_0-1} \frac{1}{t^{\alpha_{0}-1}}
\end{split}
\end{equation} 
On the other hand we can choose $\eta$ small enough such that $\beta J (1+\eta) < \gamma_{0}< \frac{1}{2}$. Now it is enough to prove that 
\begin{equation}
\limsup\E\left[ \exp\left\{ (1+\eta)\frac{\beta J}{n}\left( \sum_{i=1}^{n}\sigma_{i} \right)^2\right\} \right]<\infty.
\end{equation}
However we know that for any $t>0$,
\begin{equation}
\begin{split}
&\E\left[ \exp\left\{  \frac{t}{\sqrt{n}} \sum_{i=1}^{n} \sigma_{i}\right\} \right]\le \exp\left\{ \frac{t^2}{2} \right\}\\
& \Rightarrow
\mathbb{P}\left[ \left|\frac{1}{\sqrt{n}} \sum_{i=1}^{n} \sigma_{i} \right| >t \right] = 2 \mathbb{P}\left[ \frac{1}{\sqrt{n}} \sum_{i=1}^{n} \sigma_{i} > t \right]=2\mathbb{P}\left[ \exp \left\{ \frac{t}{\sqrt{n}} \sum_{i=1}^{n} \sigma_{i}\right\} > \exp\left\{ t^2 \right\}\right]\le 2 \exp\left\{ -\frac{t^2}{2} \right\}\\
\end{split}
\end{equation}
Here the last inequality is a straight forward application of Markov's inequality. Now 
\begin{equation}
\begin{split}
\mathbb{P}\left[ \exp\left\{\frac{\beta J(1+\eta)}{n} \left( \sum_{i=1}^{n} \sigma_{i} \right)^2  \right\} >t \right]&= \mathbb{P}\left[ \frac{\beta J (1+\eta)}{n}\left(  \sum_{i=1}^{n} \sigma_{i}\right)^2 > \log (t) \right]\\
= \mathbb{P}\left[ \left| \frac{1}{\sqrt{n}} \sum_{i=1}^{n} \sigma_{i}\right| > \sqrt{\frac{\log t}{\beta J (1+\eta)}} \right]&\le 2\exp\left\{ - \frac{\log t}{2 \beta J (1+\eta)} \right\} \le 2 \left( \frac{1}{t} \right)^{\frac{1}{2 \beta J (1+\eta)}}< 2 \left( \frac{1}{t} \right)^{\frac{1}{2\gamma_{0}}}.
\end{split}
\end{equation}
Observe that $\frac{1}{2\gamma_{0}}>1$.
Hence by argument similar to \eqref{eqn:argjoint} we have 
\begin{equation}
\limsup \E\left[ \exp\left\{ (1+\eta)\frac{\beta J}{n}\left( \sum_{i=1}^{n}\sigma_{i} \right)^2\right\} \right]< \infty.
\end{equation}
This completes the proof of uniform integrability of the random variable in the l.h.s. of \eqref{def:exp}.
As a consequence, 
\begin{equation}
\begin{split}
&\lim_{n\to \infty}\E\left[\mathbb{I}_{\sigma \in \Omega(\sigma)_{n}} \mathbb{I}_{\sigma' \in \Omega(\sigma)_{n}}\exp\left\{ \frac{2\beta^2}{n}\left( \sum_{i=1}^{n} \sigma_{i}\sigma'_{i} \right)^2  + \frac{\beta J}{n}\left( \sum_{i=1}^{n}\sigma_{i} \right)^2 + \frac{\beta J}{n}\left( \sum_{i=1}^{n}\sigma'_{i} \right)^2 \right\}\right]\\
& = \E\left[ 2\beta^2 Y_{1} + \beta J Y_{2} + \beta J Y_{3} \right]= \frac{1}{\sqrt{1-4\beta^2}}\frac{1}{1-2\beta J}.
\end{split}
\end{equation}
Plugging this into \eqref{eq:loglikemombound} we have 
\begin{equation}
\begin{split}
\lim_{n \to \infty}\E_{\mathbb{P}_{n}}\left[ \left( \frac{d\tilde{\mathbb{Q}}_{n}}{d\mathbb{P}_{n}} \right)^2  \right] &= \exp\left\{-2\beta^2\right\} \left( 1- 2\beta J \right)\frac{1}{\sqrt{1-4\beta^2}}\frac{1}{1-2\beta J}\\
&=  \exp\left\{-2\beta^2\right\} \frac{1}{\sqrt{1-4\beta^2}}\\
&= \exp \left\{ -2 \beta^2 \right\} \exp \left\{ -\frac{1}{2} \log\left(1-4\beta^2\right) \right\}\\
&= \exp \left\{ -2 \beta^2 \right\} \exp \left\{ \frac{1}{2} \sum_{k=1}^{\infty} \frac{\left(4\beta^2\right)^{k}}{k} \right\}= \exp\left\{ \sum_{k=2}^{\infty} \frac{\mu_{k}^2}{2k} \right\}
\end{split}
\end{equation}
where $\mu_{k}=(2\beta)^{k}$.
Now using Proposition \ref{prop:norcont} with $W_{n,k}= C_{n,k+1}-(n-1)\mathbb{I}_{k=1}$, we have for the sequences of measures $\tilde{\mathbb{Q}}_{n}$ and $\mathbb{P}_{n}$
\begin{equation}
\frac{d\mathbb{\tilde{Q}}_{n}}{d\mathbb{P}_{n}}\left|\mathbb{P}_{n} \right. \stackrel{d}{\to} \exp\left\{ \sum_{k=1}^{\infty} \frac{2\mu_{k+1} Z_{k}-\mu_{k+1}^2}{4(k+1)}  \right\}
\end{equation}
where $Z_{k}\sim N\left(0,2(k+1)\right)$.
Hence 
\begin{equation}
\frac{d\mathbb{\tilde{Q}}_{n}}{d\mathbb{P}_{n}}\left|\mathbb{P}_{n} \right. \stackrel{d}{\to} \exp\left\{ \sum_{k=1}^{\infty} \frac{2\mu_{k+1} Z_{k}-\mu_{k+1}^2}{4(k+1)}  \right\}.
\end{equation}
 As we have proved earlier that $\frac{d\mathbb{\tilde{Q}}_{n}}{d\mathbb{P}_{n}}- \frac{d\mathbb{Q}_{n}}{d\mathbb{P}_{n}} \left| \mathbb{P}_{n} \right. \stackrel{p}{\to} 0$, this completes the proof of the asymptotic normality of $\log\left( \frac{d\mathbb{Q}_{n}}{d\mathbb{P}_{n}} \right)\left| \mathbb{P}_{n}. \right.$

 \noindent 
 \textbf{Proof of part (2) of Theorem \ref{thm:asymptotic}:}
 Before proving part $(1)$ of Theorem \ref{thm:asymptotic}, we prove part $(2)$. Since 
 \begin{equation}
 \frac{d\mathbb{Q}_{n}}{d\mathbb{P}_{n}}
= \frac{1}{\tau_{n}} \exp\left\{-(n-1)\beta^2 + \beta J\right\} \exp\left\{ - \frac{\beta}{\sqrt{n}} \sum_{i=1}^{n} A_{i,i}  -\beta J'\right\}Z_{n}(\beta),
 \end{equation} 
 in order to prove part $(2)$ of Theorem \ref{thm:asymptotic}, we need to prove that 
 \begin{equation}\label{eqn:toprove}
 \log \left( \frac{d\mathbb{Q}_{n}}{d\mathbb{P}_{n}} \right)- \sum_{k=2}^{m_{n}} \frac{2\mu_{k} \left( C_{n,k}- (n-1)\mathbb{I}_{k=2} \right)-\mu_{k}^2}{4k} \left| \mathbb{P}_{n} \right. \stackrel{p}{\to} 0.
 \end{equation}
 We at first prove the result analogous to \eqref{eqn:toprove} for $\log\left(\frac{d\tilde{\mathbb{Q}}_{n}}{d\mathbb{P}_{n}}\right)$. \eqref{eqn:toprove} then follows from the fact that $\frac{d\mathbb{Q}_{n}}{d\mathbb{P}_{n}} - \frac{d\tilde{\mathbb{Q}}_{n}}{d\mathbb{P}_{n}}\left| \mathbb{P}_{n} \right.\stackrel{p}{\to} 0$ and an application of continuous mapping theorem.
 
By \eqref{eq:janson-decomp},
for any given $\epsilon,\delta >0$ there exists $K=K(\epsilon,\delta)$ and for any subsequence $n_l$ there exists a further subsequence $n_{l_q}$ such that 
\begin{equation}\label{bdd:liketoseriesIII}
\mathbb{P}_{n_{l_q}}\left( \left| \log(\frac{d\tilde{\mathbb{Q}}_{n_{l_{q}}}}{d\mathbb{P}_{n_{l_{q}}}}) - 
\sum_{k=2}^{K} \frac{2\mu_{k}(C_{n_{l_q},k}-(n-1)\mathbb{I}_{k=2})-\mu_{k}^{2}}{4k} \right|\ge \frac{\epsilon}{2} \right) \le  \frac{\delta}{2}.
\end{equation} 
Now choose $K'\ge K$ such that 
\begin{equation}
\sum_{K'+1}^{\infty}\frac{\mu_{k}^{2}}{2k}\le \max\left\{\frac{\delta\epsilon^{2}}{100},\frac{\epsilon}{100}\right\}.
\end{equation}
For any $K'< k_{1}<k_{2}<m_{n}= o(\sqrt{\log n})$, the proof of Proposition \ref{prop:signdistr} implies that $\E_{\mathbb{P}_{n}}\left[ C_{n,k_{1}} \right]=0$, $\Cov(C_{n,k_{1}},C_{n,k_{2}})=0$ and $\Var(C_{n,k_{i}})= 2k_{i}(1+ O({k_{i}^{2}}/{n}))$ for $i \in \{1,2 \}.$
So 
\begin{equation}
\Var\left(\sum_{k=K'+1}^{m_{n_{l_q}}} \frac{2\mu_{k} C_{n_{l_q},k}-\mu_{k}^{2}}{4k}\right)
= 
(1+o(1))\sum_{k=K'+1}^{m_{n_{l_q}}}\frac{\mu_{k}^{2}}{2k}\le \frac{\delta\epsilon^{2}}{100}.
\end{equation}
Now for large values of $n_{l_q}$,
\begin{equation}
	\label{bdd:series}
\begin{split}
&\mathbb{P}_{n_{l_q}}\left( \left|\sum_{k=K+1}^{m_{n_{l_q}}} \frac{2\mu_{k}C_{n_{l_q},k}}{4k}\right| \ge \frac{\epsilon}{4}\right) \le \frac{16 \delta\epsilon^{2}}{100\epsilon^{2}},
\qquad \mbox{and so} \\
& \mathbb{P}_{n_{l_q}} \left( \left|\sum_{k=K+1}^{m_{n_{l_q}}} \frac{2\mu_{k} C_{n_{l_q},k} -\mu_{k}^{2}}{4k}  \right| \ge \frac{\epsilon}{4}+ \frac{\epsilon}{100}  \right)
\le\frac{16 \delta\epsilon^{2}}{100\epsilon^{2}}.
\end{split}
\end{equation}

\noindent Plugging in the estimates of \eqref{bdd:liketoseriesIII} and \eqref{bdd:series} we have for all large values of $n_{l_q}$,
\begin{equation}
	\label{eq:subseq-conv}
\mathbb{P}_{n_{l_q}}\left(  \left| \log\left(\frac{d\tilde{\mathbb{Q}}_{n_{l_{q}}}}{d\mathbb{P}_{n_{l_{q}}}}\right) - \sum_{k=1}^{m_{n_{l_q}}} \frac{2\mu_{k}(C_{n_{l_q},k}- (n-1)\mathbb{I}_{k=2})-\mu_{k}^{2}}{4k}
\right| \ge \epsilon \right)
\le \delta.
\end{equation}
Since \eqref{eq:subseq-conv} occurs to any subsequence and any $(\epsilon,\delta)$ pair, this completes the proof.

\noindent 
\textbf{Proof of part (1) of Theorem \ref{thm:asymptotic}:} 
Consider the random variable 
\begin{equation}
M:= W+ \sum_{k=1}^{\infty}\frac{2\mu_{k+1}Z_{k}-\mu_{k+1}^2}{4(k+1)}
\end{equation} 
where $W\sim N(0, \beta^2)$ and is independent of the random variable 
\begin{equation}
\sum_{k=1}^{\infty}\frac{2\mu_{k+1}Z_{k}-\mu_{k+1}^2}{4(k+1)}. 
\end{equation} 
Observe that from the proof of part $(2)$ we have 
\begin{equation}
\begin{split}
&\log\left(Z_{n}(\beta)\right) + \frac{1}{2} \log\left(1- 2\beta J\right)- (n-1)\beta^2 + \beta (J  -J') - \beta C_{n,1}\\
& ~~~~~~~~~~ - \sum_{k=2}^{m_{n}}\frac{2\mu_{k}\left(C_{n,k}- (n-1)\mathbb{I}_{k=2}\right)- \mu_{k}^2}{4k}\left| \mathbb{P}_{n} \right.\stackrel{p}{\to} 0.
\end{split}
\end{equation}
So it is enough to prove that 
\begin{equation}
\beta C_{n,1} + \sum_{k=2}^{m_{n}}\frac{2\mu_{k}\left(C_{n,k}- (n-1)\mathbb{I}_{k=2}\right)- \mu_{k}^2}{4k} \stackrel{d}{\to} N\left( \beta^2+\frac{1}{4}\log(1-4\beta^2), -\beta^2 - \frac{1}{2}\log(1-4\beta^2) \right).
\end{equation}
On the other hand for any fixed $K$, 
\begin{equation}
\beta C_{n,1} + \sum_{k=2}^{K}\frac{2\mu_{k}\left(C_{n,k}- (n-1)\mathbb{I}_{k=2}\right)- \mu_{k}^2}{4k}\left| \mathbb{P}_{n} \right.\stackrel{d}{\to} W + \sum_{k=1}^{K-1}\frac{2\mu_{k+1}Z_{k}-\mu_{k+1}^2}{4(k+1)}.
\end{equation} 
Since all the random variables $\beta C_{n,1}$, $\sum_{k=2}^{m_{n}}\frac{2\mu_{k}\left(C_{n,k}- (n-1)\mathbb{I}_{k=2}\right)- \mu_{k}^2}{4k}$ and $\sum_{k=2}^{K}\frac{2\mu_{k}\left(C_{n,k}- (n-1)\mathbb{I}_{k=2}\right)- \mu_{k}^2}{4k}$ have uniformly bounded second moments, they are tight. Hence we have any of their linear combination is also tight. Hence given any subsequence $n_{l} $ there exists a further subsequence $n_{l_{q}}$ such that 
\begin{equation}
\beta C_{n_{l_{q}},1} + \sum_{k=2}^{m_{n_{l_{q}}}}\frac{2\mu_{k}\left(C_{n_{l_{q}},k}- (n_{l_{q}}-1)\mathbb{I}_{k=2}\right)- \mu_{k}^2}{4k} \left| \mathbb{P}_{n_{l_q}} \right. \stackrel{d}{\to} M\{ n_{l_{q}} \}.
\end{equation}
Here the notation $M\{ n_{l_{q}} \}$ means that the limiting distribution might possibly depend on the choice of the sub-sequence.
On the other hand for every fixed $K$ there is a further subsequence $n_{{l_{q}}_{m}}$ (possibly dependent on $K$) such that 
\begin{equation}
\begin{split}
&\left(\beta C_{n_{{l_{q}}_{m}},1} + \sum_{k=2}^{m_{n_{{l_{q}}_{m}}}}\frac{2\mu_{k}\left(C_{{n_{l_{q}}}_{m},k}- (n_{{l_{q}}_{m}}-1)\mathbb{I}_{k=2}\right)- \mu_{k}^2}{4k},\beta C_{n_{{l_{q}}_{m}},1} + \sum_{k=2}^{K}\frac{2\mu_{k}\left(C_{{n_{l_{q}}}_{m},k}- (n_{{l_{q}}_{m}}-1)\mathbb{I}_{k=2}\right)- \mu_{k}^2}{4k}\right)\left| \mathbb{P}_{n_{{l_{q}}_{m}}}\right.\\
&~~~~~~~~~~~~~~~~~~~~~~~~~~~~~~~~~~~~~~~~~~~~~~~~~~~~~~~~~~~~~~~~~~~~~~~~~~~~~~~~~~~~~~~~~\stackrel{d}{\to}\left( M_{1},M_{2,K}  \right).
\end{split}
\end{equation}
where $M_{1}\stackrel{d}{=}M\{  n_{l_{q}}\}$ and $M_{2,K}\stackrel{d}{=} W+ \sum_{k=1}^{K-1}\frac{2\mu_{k+1}Z_{k}-\mu_{k+1}^2}{4(k+1)}$.
Hence
\begin{equation}
\sum_{k=K+1}^{m_{n_{{l_{q}}_{m}}}}\frac{2\mu_{k}\left(C_{n_{{l_{q}}_{m}},k}- (n_{{l_{q}}_{m}}-1)\mathbb{I}_{k=2}\right)- \mu_{k}^2}{4k} \left| \mathbb{P}_{n_{{l_{q}}_{m}}} \right. \stackrel{d}{\to} M_{1}-M_{2,K}.
\end{equation}
On the other hand by Fatou's lemma,
\begin{equation}\label{balchal}
\liminf\E_{\mathbb{P}_{n_{{l_{q}}_{m}}}}\left[\left( \sum_{k=K+1}^{m_{n}}\frac{2\mu_{k}\left(C_{n_{{l_{q}}_{m}},k}- (n-1)\mathbb{I}_{k=2}\right)- \mu_{k}^2}{4k} \right)^2\right] \ge  \E\left[(M_{1}-M_{2,K})^2\right].
\end{equation}
We know that for large enough value of $n_{{l_{q}}_{m}}$,
\begin{equation}
\begin{split}
&\E_{\mathbb{P}_{n_{{l_{q}}_{m}}}}\left[\left( \sum_{k=K+1}^{m_{n_{{l_{q}}_{m}}}}\frac{2\mu_{k}\left(C_{n_{{l_{q}}_{m}},k}- (n-1)\mathbb{I}_{k=2}\right)- \mu_{k}^2}{4k} \right)^2\right] = \Var\left( \sum_{k=K+1}^{m_{n_{{l_{q}}_{m}}}} \frac{2\mu_{k}\left(C_{n_{{l_{q}}_{m}},k}- (n-1)\mathbb{I}_{k=2}\right)}{4k}\right) + \left(\sum_{k=K+1}^{m_{n_{{l_{q}}_{m}}}} \frac{\mu_{k}^2}{4k}\right)^2\\
&= (1+o(1))\sum_{k=K+1}^{m_{n_{{l_{q}}_{m}}}} \frac{\mu_{k}^2}{2k} + \left(\sum_{k=K+1}^{m_{n_{{l_{q}}_{m}}}} \frac{\mu_{k}^2}{4k}\right)^2.
\end{split}
\end{equation}
Here we have used the identity $\E[X^2]= \Var[X]+ \left(\E[X]\right)^2$.
Given any $\epsilon >0$,
we now choose $K$ large enough so that 
$
\sum_{k=K+1}^{\infty} \frac{\mu_{k}^2}{2k} \le \epsilon,
$
implying $\E\left[ (M_{1}-M_{2,K})^2\ \right]\le \epsilon + \epsilon^2/4$.
Hence 
the r.h.s. of \eqref{balchal} converges to $0$ as $K \to \infty$. This implies $W_{2}\left( F^{M_{1}}, F^{M_{2,K}} \right)\to 0$ as $K \to \infty$. Here $F^{M_{1}}$ and $F^{M_{2,K}}$ denote the distribution functions of $M_{1}$ and $M_{2,K}$ respectively. As a consequence, we have 
\begin{equation}
W+ \sum_{k=1}^{K-1}\frac{2\mu_{k+1}Z_{k}-\mu_{k+1}^2}{4(k+1)} \stackrel{d}{\to} M\{n_{l_{q}}\}.
\end{equation}
Hence $M\{ n_{l_{q}} \}\stackrel{d}{=} W+ \sum_{k=1}^{\infty}\frac{2\mu_{k+1}Z_{k}-\mu_{k+1}^2}{4(k+1)} $ which does not depend on the specific choice of the subsequence $\{ n_{l_{q}} \}$. 
This concludes the proof. 
\hfill{$\square$}\\

\appendix
We now give proofs of  Theorem \ref{Thm:approximation} and Proposition \ref{prop:signdistr}
\section{Proof of Theorem \ref{Thm:approximation}}\label{sec:thmapp}
Although the proof of Theorem \ref{Thm:approximation} is similar to the proof of Theorem 3.4 in \citet{Banerjee2017}, we sketch the main details here for the shake of completeness. In order to complete the proof we first need some preliminary notations, definitions and some results. All these definitions can be found in \citet{Banerjee2017}. One might have a look at Section A. \\
The proof of Theorem \ref{Thm:approximation} is divided into two parts depending upon $k$ being even or odd. We analyze each case separately. The case $k$ odd is almost similar to the case considered in \citet{Banerjee2017}. However the case $k$  even is easier here than the case considered in \citet{Banerjee2017}. This due to the fact that when $k$ is even there are words with $l(w)=2k+1$ such that $G_{w}$ is a tree. This creates additional complications. We discuss these things in details now.\\
As a general discussion, the fundamental idea is to show that $\Tr\left[ \left( \frac{1}{\sqrt{n}} \right)^{2k+1}\tilde{A}^{2k+1} \right]$ can be written as a linear combination of $C_{n,r}$'s where $r$ is odd. In particular, one shows that the mean of $\Tr\left[ \left( \frac{1}{\sqrt{n}} \right)^{2k+1}\tilde{A}^{2k+1} \right]$ is approximately $0$ and the main contribution to the variance comes from a class of closed words which corresponds to some multiple of the $C_{n,r}$'s. One can show that the contribution by all the other words are negligible. The Chebyshev polynomial approximation is obtained by inverting this relation. On the other hand when $k$ is even, $\Tr\left[\Tr\left[ \left( \frac{1}{\sqrt{n}} \right)^{2k}\tilde{A}^{2k} \right]\right]$ can be written as a linear combination of $C_{n,r}$'s plus an additional term which has two components. The first component is approximately equal to $k\psi_{2k}\left[ \frac{1}{n}\Tr[\tilde{A}_{n}^{2}]- \E\left[ \frac{1}{n}\Tr[\tilde{A}_{n}^{2}] \right] \right]$ and the second component gives a non-trivial contribution to the mean keeping the variance unchanged. These terms correspond to words in $\mathcal{W}_{1,2k}$ and $\mathcal{W}_{3,2k}$ in the proof. Finally the Chebyshev polynomial approximation still holds due to the fact 
\begin{equation}
\sum_{r=1}^{k} P_{2k}[2r]r\psi_{2r}=0.
\end{equation}    
This cancels out the term corresponding to $k\psi_{2k}\left[ \frac{1}{n}\Tr[\tilde{A}_{n}^{2}]- \E\left[ \frac{1}{n}\Tr[\tilde{A}_{n}^{2}] \right] \right]$.

\noindent
We now elaborate the above discussion in some details. 
\begin{enumerate}[(i)]
\item \textbf{$k$ is odd:}In this part to avoid confusion we shall use the notation $2k+1$ instead of using the terminology $k$ being odd. The proof is completed by first showing that $\Tr\left[ \left( \frac{1}{\sqrt{n}} \right)^{2k+1}\tilde{A}^{2k+1} \right]$ can be approximated as a linear combination of the cycles $C_{n,i}$'s. Then one inverts the relation to get cycles as a linear combination of the traces. \\
We at first write down the following identity:
\begin{equation}
\begin{split}
\Tr\left[ \left( \frac{1}{\sqrt{n}} \right)^{2k+1}\tilde{A}^{2k+1} \right]& = \left( \frac{1}{n} \right)^{\frac{2k+1}{2}}\sum_{w~|~ \text{closed} ~ l(w)=2k+2} X_{w}
\end{split}
\end{equation}
where the graph $G_{w}$ has no self loops and for any closed word $w=(i_{0},i_{1},\ldots,i_{2k},i_{0})$ we define 
\begin{equation}
X_{w}:= A_{i_{0},i_{1}}\ldots A_{i_{2k},i_{0}}.
\end{equation}
By a direct use of the parity principle (Lemma A.1 in \citet{Banerjee2017}) we get that $G_{w}$ can not be a tree.
We now divide the class of words with $l(w)=2k+2$ in to two parts. In the first part
  we consider the cases when $G_{w}$ is an unicyclic graph such that each edge in the bracelet is repeated exactly once in particular $w \in \mathfrak{W}_{2k+2,r,t}$ for some $r$ odd and $2t -r= m$ (to be denoted by $\mathfrak{W}_{2k+2,r}$) and we denote the complement by $\mathcal{W}_{1,2k+1}$. In particular, we write 
\begin{equation}\label{eq:deom}
\begin{split}
 &\left( \frac{1}{n} \right)^{\frac{2k+1}{2}}\sum_{w~|~ \text{closed} ~ l(w)=2k+2} X_{w}\\
 & = \left( \frac{1}{n} \right)^{\frac{2k+1}{2}}\sum_{r=3}^{2k+1}\sum_{w \in \mathfrak{W}_{2k+2,r,t}} X_{w}+  \left( \frac{1}{n} \right)^{\frac{2k+1}{2}} \sum_{w \in \mathcal{W}_{1,2k+1}} X_{w}.
\end{split}
\end{equation} 
Our fundamental goal is to show that
\begin{equation}
\E\left[\left(\left( \frac{1}{n} \right)^{\frac{2k+1}{2}} \sum_{w \in \mathcal{W}_{1,2k+1}} X_{w}\right)^2\right] \to 0
\end{equation} and 
\begin{equation}\label{eq:cyctotr}
\E\left[ \left(\left( \frac{1}{n} \right)^{\frac{2k+1}{2}}\sum_{r=3~|~ r ~ \text{odd}}^{2k+1}\sum_{w \in \mathfrak{W}_{2k+2,r}} X_{w} - \sum_{r=3~|~ r~ \text{odd}}^{2k+1}f(2k+1,r)\frac{2k+1}{r}C_{n,r}\right)^2 \right]\to 0.
\end{equation}
We at first analyze the second part of \eqref{eq:deom}. 
\begin{equation}
\begin{split}
&\E\left[\left(\left( \frac{1}{n} \right)^{\frac{2k+1}{2}} \sum_{w \in \mathcal{W}_{1,2k+1}} X_{w}\right)^2\right]\\
&=\left( \frac{1}{n} \right)^{2k+1} \sum_{w, x \in \mathcal{W}_{1,2k+1}} \E[X_{w}X_{x}].
\end{split}
\end{equation}
Now $\E[X_{w}X_{x}]=0$ unless all the edges in $a=[w,x]$ are repeated at least twice. Observe that by Holder's inequality and sub Gaussianity we have $\E[\left| X_{w} X_{x}\right|]\le \left|A_{1,2}\right|^{4k+2}\le \left( C_{1}k \right)^{C_{2}k}$. Here $C_{1}$ and $C_{2}$ are two deterministic constants. We divide this case into two further sub cases. First of all $w$ and $x$  shares an edge in particular $a$ is an weak CLT sentence If $\#V_{a}=t$ by Lemma A.5 in \citet{Banerjee2017}, the number of such $a$ is bounded by $2^{4k+4}\left( C_{3}(4k+4) \right)^{C_{4}}\left( 4k+4 \right)^{3(4k+4-2t)} n^{t}$. By Proposition A.6 in \citet{Banerjee2017}, $w,x \in \mathcal{W}_{1,2k+1}$, $a$ can not be a CLT word pair. Hence $\#V_{a} < 2k+1$.  Now consider the case when $w$ and $x$ don't share an edge. In this case $\#V_{w}\le \#E_{w}< \frac{2k+1}{2}$ and for any $t_{1}$ such that $\#V_{w}=t_{1}$ we have by Lemma A.4 in \citet{Banerjee2017} we have the cardinality of such $w$ is bounded by $2^{2k+1}(2k+1)^{3(2k-2t_{1}+3)}n^{t_{1}}$. Similarly the number of words $x$ such that for any $t_{2}$ such that $\#V_{x}=t_{2}$ is bounded by $2^{2k+1}(2k+1)^{3(2k-2t_{2}+3)}n^{t_{2}}$. As a consequence,
\begin{equation}
\begin{split}
&\E\left[\left(\left( \frac{1}{n} \right)^{\frac{2k+1}{2}} \sum_{w \in \mathcal{W}_{1,2k+1}} X_{w}\right)^2\right]\\
&\le \left( C_{1}k \right)^{C_{2}k}\left[\sum_{t=1}^{2k} 2^{4k+4}\left( C_{3}(4k+4) \right)^{C_{4}}\left( 4k+4 \right)^{3(4k+4-2t)}\left(\frac{1}{n}\right)^{2k+1-t}\right.\\
&~~~~~\left.+ \sum_{t_{1}=1}^{k}\sum_{t_{2}=1}^{k}2^{4k+2}(2k+1)^{3(4k-2t_{2}-2t_{1}+6)}\left(\frac{1}{n}\right)^{2k+1-t_{1}-t_{2}}\right]\\
&~~~~ \to 0
\end{split}
\end{equation}
whenever $k=o\left( \sqrt{\log n} \right)$.  Hence we can neglect the second term in \eqref{eq:deom}. Now we prove \ref{eq:cyctotr}.
\\
First of all for any $w \in \mathfrak{W}_{2k+2,r}$, $\E[X_{w}]=0$ and if $w_{1} \in \mathfrak{W}_{2k+2,r_{1}}$ and $w_{2} \in \mathfrak{W}_{2k+2,r_{2}}$ where $r_{1}\neq r_{2}$ $\E[X_{w_{1}}X_{w_{2}}]=0$. Hence 
\begin{equation}
\begin{split}
&\E\left[\left(\left( \frac{1}{n} \right)^{\frac{2k+1}{2}}\sum_{r=3~|~ r ~ \text{odd}}^{2k+1}\sum_{w \in \mathfrak{W}_{2k+2,r}} X_{w}\right)^2\right]= \sum_{r=3~|~ r ~ \text{odd}}\Var\left[ \left( \frac{1}{n} \right)^{\frac{2k+1}{2}}\sum_{w \in \mathfrak{W}_{2k+2,r}} X_{w} \right]\\
&= \left(1+O\left( \frac{k^2}{n} \right) \right)\sum_{r=3~|~ r ~ \text{odd}} f^{2}(2k+1,r)\frac{2(2k+1)^{2}}{r}.
\end{split}
\end{equation}
Here we have used Proposition A.6 in \citet{Banerjee2017} and Lemma A.10 in \citet{Banerjee2017}. On the other hand by proof of Proposition \ref{prop:signdistr} we have 
\begin{equation}
\begin{split}
&\E\left[  \left(\sum_{r=3~|~ r~ \text{odd}}^{2k+1}f(2k+1,r)\frac{2k+1}{r}C_{n,r}\right)^2 \right]\\
&= \left(1+O\left( \frac{k^2}{n} \right) \right)\sum_{r=3~|~ r ~ \text{odd}} f^{2}(2k+1,r)\frac{2(2k+1)^2}{r}.
\end{split}
\end{equation}
Now 
\begin{equation}
\begin{split}
&\E\left[ \left(\left( \frac{1}{n} \right)^{\frac{2k+1}{2}}\sum_{r=3~|~ r ~ \text{odd}}^{2k+1}\sum_{w \in \mathfrak{W}_{2k+2,r}} X_{w}\right) \left( \sum_{r=3~|~ r~ \text{odd}}^{2k+1}f(2k+1,r)\frac{2k+1}{r}C_{n,r} \right) \right]\\
&= \sum_{r=3~|~ r ~ \text{odd}}f(2k+1,r)\frac{2k+1}{r}\Cov\left[ \left( \frac{1}{n} \right)^{\frac{2k+1}{2}}\sum_{w \in \mathfrak{W}_{2k+2,r}} X_{w},C_{n,r} \right]\\
&= \sum_{r=3~|~ r ~ \text{odd}}\left( 1+ O\left( \frac{k^2}{n} \right) \right) f(2k+1,r)\frac{2k+1}{r} 2f(2k+1,r) \frac{2k+1}{r}\\
&= \sum_{r=3~|~ r ~ \text{odd}}\left( 1+ O\left( \frac{k^2}{n} \right) \right) f^{2}(2k+1,r)\frac{2(2k+1)^2}{r}
\end{split}
\end{equation}
by Lemma A.10 in \citet{Banerjee2017}. As a consequence,
\begin{equation}
\begin{split}
&\E\left[ \left(\left( \frac{1}{n} \right)^{\frac{2k+1}{2}}\sum_{r=3~|~ r ~ \text{odd}}^{2k+1}\sum_{w \in \mathfrak{W}_{2k+2,r}} X_{w} - \sum_{r=3~|~ r~ \text{odd}}^{2k+1}f(2k+1,r)\frac{2k+1}{r}C_{n,r}\right)^2 \right]\\
&= \sum_{r=3~|~ r \text{odd}} O\left( \frac{k^2}{n} \right)f^{2}(2k+1,r)\frac{2(2k+1)^2}{r}\\
&= O\left(\frac{\text{Poly}(k)2^{4k}}{n}\right)\to 0.
\end{split}
\end{equation}
Here $\text{Poly}(k)$ is a deterministic polynomial in $k$. It is known that for example see \citet{Lang} that $f(m,r)\frac{m}{r}=\binom{m}{\frac{m+r}{2}}$ when $m$ and $r$ have same parity. By comparing coefficient it can be checked that 
\begin{equation}
\left(
\begin{array}{lllll}
1&0 &\ldots & 0 & 0\\
\binom{3}{2}&1&\ldots&0& 0\\
\vdots & \ddots & \ldots & 1& 0\\
\binom{2k+1}{k+1} & \binom{2k+1}{k+2} & \ldots & \binom{2k+1}{2k} & 1
\end{array}
\right)^{-1}= D.
\end{equation}
Here $D$ is a lower triangular matrix and $D_{i,j}= P_{2i+1}[2j+1]$ where $P_{j}[i]$ is the coefficient of $z^{i}$ in $\left( z + \frac{1}{z} \right)^{j}$. Now 
\begin{equation}\label{eq:quadbound}
\begin{split}
&\E\left[\left(C_{n,2k+1}- \Tr\left[ P_{2k+1}\left[\frac{1}{\sqrt{n}}\tilde{A}\right] \right]\right)^2\right]\\
& \le \sum_{r=3~|~ r ~ \text{odd}}^{2k+1}\E\left[ \left(C_{n,2r+1}- \Tr\left[ P_{2r+1}\left[\frac{1}{\sqrt{n}}\tilde{A}\right] \right]\right)^2 \right]\\
& = \E\left[ \zeta'D' D \zeta \right]\\
&\le \sup_{i,j}| D'D_{i,j}| \E\left[ \sum_{i=1}^{k}\zeta_{i}^{2} + 2\sum_{i<j}\left|\zeta_{i}\zeta_{j}\right| \right]\\
&\le  \sup_{i,j}| D'D_{i,j}|k \E\left[ \sum_{i=1}^{k} \zeta_{i}^{2} \right]\to 0.
\end{split}
\end{equation}
Here 
\begin{equation}
\zeta=\left(
\begin{array}{l}
0  \\
\Tr\left[\left(\frac{1}{\sqrt{n}}\right)^{3}\tilde{A}^{3}\right]- \binom{3}{2}C_{n,3}\\
 \vdots\\
 \Tr \left[ \left(\frac{1}{\sqrt{n}}\right)^{2k+1}\tilde{A}^{2k+1} \right]- \sum_{r=3~|~ r ~\text{odd}}^{2k+1} \binom{2k+1}{\frac{2k+1+r}{2}} C_{n,r}
\end{array}
\right)
\end{equation}
and the last step of \eqref{eq:quadbound} follows from the fact that the coefficients of the Chebyshev polynomial grows at most exponentially.
This completes the proof of the odd case. 
\item \textbf{$k$ is even:} Here again to avoid confusion we use the notation $2k$ instead of using the term $k$ even. Like the $2k+1$ case we at first write 
\begin{equation}\label{eq:sumeven}
\begin{split}
& \Tr\left[ \left(\frac{1}{\sqrt{n}}\right)^{2k}\tilde{A}^{2k} \right] = \left( \frac{1}{n} \right)^{k}\sum_{w~|~ \text{closed} ~ l(w)=2k+1} X_{w}.
\end{split}
\end{equation}
However unlike the $2k+1$ case, in this case there are words $w$ such that $G_{w}$ is a tree. Keeping this in mind, we decompose $\Tr\left[ \left(\frac{1}{\sqrt{n}}\right)^{2k}\tilde{A}^{2k} \right]$ in the following way:
\begin{equation}
\begin{split}
\Tr\left[ \left(\frac{1}{\sqrt{n}}\right)^{2k}\tilde{A}^{2k} \right]=  \left( \frac{1}{n} \right)^{k} \left[ \sum_{w \in \mathcal{W}_{1,2k}} X_{w} + \sum_{w \in \mathcal{W}_{2,2k}} X_{w}+ \sum_{w \in \mathcal{W}_{3,2k}}X_{w}+ \sum_{w \in \mathcal{W}_{4,2k}}X_{w}\right]
\end{split}
\end{equation}
where 
\begin{enumerate}
\item $\mathcal{W}_{1,2k}$ is the collection of all Wigner words(See Definition A.3 in \citet{Banerjee2017}). In this case $G_{w}$ is a tree and each is traversed exactly twice.
\item $\mathcal{W}_{2,2k}:= \cup_{r=4~:~ r ~ \text{even}}^{2k}\mathfrak{W}_{2k+1,r}$.
\item $\mathcal{W}_{3,2k}$ is the collection of words such that $G_{w}$ is either an unicyclic graph with all edges repeated exactly twice or $G_{w}$ is a tree with exactly one edge repeated exactly four times and all other edges are repeated exactly twice.
\item $\mathcal{W}_{4,2k}$ is collection of all other words. 
\end{enumerate}
We at first prove that 
\begin{equation}\label{eq:decomsecondeven}
\begin{split}
\left( \frac{1}{n} \right)^{2k}\E\left[\left(\sum_{w \in \mathcal{W}_{4,2k}}X_{w}\right)^2 \right]\to 0.
\end{split}
\end{equation}
Note that 
\begin{equation}
\E\left[\left(\sum_{w \in \mathcal{W}_{4,2k}}X_{w}\right)^2 \right] = \sum_{w,x \in \mathcal{W}_{4,2k}} \E[X_{w}X_{x}].
\end{equation}
Alike the odd case $\E[X_{w}X_{x}]=0$ unless all the edges in the sentence $a=[w,x]$ are repeated at least twice. Now we consider two cases firstly when the sentence $a$ is a weak CLT sentence i.e. the graphs $G_w$ and $G_x$ share an edge. In this case we prove that $\#V_{a}<2k$. We do a case by case analysis here. Since all the edges in $a$ are repeated at least twice $\#E_{a}\le 2k$. Whenever $\#E_{a}=2k$, there are two cases when $\#V_{a}\ge 2k$. Firstly the graph $G_{a}$ is a tree and each edge is repeated exactly twice. This is impossible since otherwise in both the words $w$ and $x$ there will be an edge which are repeated exactly once however both $G_{w}$ and $G_{x}$ are tree. Now consider the other case when $\#E_{a}=\#V_{a}=2k$ in this case $G_{w}$ and $G_{x}$ are both unicyclic graphs with common bracelet. This is impossible by definition of $\mathcal{W}_{4,2k}$. Finally the when $G_{w}$ is a tree and $\#E_{a}=2k-1$  is also impossible since in this case the only possibility is both $w, x \in \mathcal{W}_{1,2k}$. As a consequence $\#V_{a}<2k$.  Now we consider the other case when $G_{w}$ and $G_{x}$ don't share an edge. In this case we have all the edges in both $G_{w}$ and $G_{x}$ are repeated at least twice. However all the words of such type having $\#V_{w}\ge k$ or $\#V_{x}\ge k$ are covered in $W_{1,2k}$, $W_{2,2k}$ and $W_{3,2k}$. Hence both $\#V_{w}$ and $\#V_{x}$ are strictly less than $k$. Now arguments exactly similar to the analysis of the second term of \eqref{eq:deom} proves \eqref{eq:decomsecondeven}. 
By arguments similar to the odd case, it can be proved that 
\begin{equation}
\begin{split}
\E\left[\left(\left( \frac{1}{n} \right)^{k}\sum_{w \in \mathcal{W}_{2,2k}}X_{w} - \sum_{r=4 ~ r ~ even}^{2k} f(2k,r)\frac{2k}{r}C_{n,r}\right)^{2}\right]\le O\left(\frac{\text{Poly}(k)2^{4k}}{n}\right).
\end{split}
\end{equation}
Next we prove that $\left( \frac{1}{n} \right)^{k} \sum_{w \in \mathcal{W}_{1,2k}}X_{w}- \E\left[ \left( \frac{1}{n} \right)^{k} \sum_{w \in \mathcal{W}_{1,2k}}X_{w} \right]$ can be approximated by $k\psi_{2k}\left[ \frac{1}{n}\Tr[\tilde{A}_{n}^{2}]- \E\left[ \frac{1}{n}\Tr[\tilde{A}_{n}^{2}] \right] \right]$. This is done by second moment calculation. Observe that 
\begin{equation}\label{eq:secontr}
\begin{split}
&\E\left[ \left( \left( \frac{1}{n} \right)^{k} \sum_{w \in \mathcal{W}_{1,2k}}X_{w}- \E\left[ \left( \frac{1}{n} \right)^{k} \sum_{w \in \mathcal{W}_{1,2k}}X_{w} \right]- k\psi_{2k}\left[ \frac{1}{n}\Tr[\tilde{A}_{n}^{2}]- \E\left[ \frac{1}{n}\Tr[\tilde{A}_{n}^{2}] \right] \right] \right)^2 \right]\\
&= \Var \left[ \left( \frac{1}{n} \right)^{k} \sum_{w \in \mathcal{W}_{1,2k}}X_{w} \right]-2 k\psi_{2k}\Cov \left[ \left( \frac{1}{n} \right)^{k} \sum_{w \in \mathcal{W}_{1,2k}}X_{w},\frac{1}{n}\Tr[\tilde{A}_{n}^{2}] \right]+ k^{2}\psi_{2k}^{2}\Var\left[ \frac{1}{n}\Tr[\tilde{A}_{n}^{2}]  \right]
\end{split}
\end{equation}
Now we analyze each term separately. First 
\begin{equation}\label{eq:varw1}
\begin{split}
&\Var\left[ \left( \frac{1}{n} \right)^{k} \sum_{w \in \mathcal{W}_{1,2k}}X_{w} \right]\\
&= \left( \frac{1}{n} \right)^{2k} \sum_{w,x \in \mathcal{W}_{1,2k}}\E\left[ \left(X_{w}-\E[X_{w}]\right)\left( X_{x}- \E[X_{x}] \right) \right]\\
&= \left( \frac{1}{n} \right)^{2k} \sum_{w,x\in \mathcal{W}_{1,2k};a=[w,x]~\text{weak CLT sentence}}\E\left[ \left(X_{w}-\E[X_{w}]\right)\left( X_{x}- \E[X_{x}] \right) \right]
\end{split}
\end{equation}
Since $\# \left( E_{w} \cap E_{x} \right)\ge 1$, we have $\#E_{a}\le 2k-1$. Hence $\#V_{a}\le 2k$. The equality $\#V_{a}=2k$ occurs when $w$ and $x$ share exactly one edge. We at first fix a word $w$, then the number of $x$'s such that $\#V_{a}=2k$ holds can be enumerated as follows. In the graph $G_{w}$ there are $k$ distinct edges we chose one of them which shares the edge with the word $x$. After fixing this edge we choose the equivalence class of the word $x$. There are $\psi_{2k}$ many of them and once an equivalence class is fixed, there are $k$ choices for the edge in $G_{x}$ to be shared. Now once this edge is also fixed there are two choices such that this edge is same as the chosen edge in $G_{w}$, one in the same order as the edge in $G_{w}$ and other in the reverse order. Now fixing all these choices there are $n^{k-1}\left( 1+ O \left( \frac{k^2}{n} \right) \right)$ choices for other vertices in $G_{x}$. As a consequence, given $w$ there are $2k^{2}\psi_{2k}n^{k-1}\left( 1+ O \left( \frac{k^2}{n} \right) \right)$ choices of $x$ such that $a$ is a weak CLT sentence and $\#V_{a}=2k$. Finally there are $\psi_{2k}n^{k+1}\left( 1+ O \left( \frac{k^2}{n} \right) \right)$. Combining all these, we have the number of $a$ such that $\#V_{a}=2k$ is given by $2k^{2}\psi_{2k}^{2}n^{2k}\left( 1+ O \left( \frac{k^2}{n} \right) \right)$. In this case $\E\left[ \left(X_{w}-\E[X_{w}]\right)\left( X_{x}- \E[X_{x}] \right) \right]=2$ (by Gaussianity). Finally for all the cases the number of $a$'s such that $\#V_{a}=t$ is bounded by $n^{t}2^{4k+2}\left( C_{1}(4k+2) \right)^{C_{2}}\left( 4k+2 \right)^{3\left( 4k+2 -2t\right)}$. Finally for any $a$, $\E\left[ \left|\left(X_{w}-\E[X_{w}]\right)\left( X_{x} - \E[X_{x}]\right)\right| \right]\le \left( C_{3}k \right)^{C_{4}k}$ for some deterministic constant $C_{3}$ and $C_{4}$. Plugging all these estimates in \eqref{eq:varw1} we have, 
\begin{equation}
\begin{split}
& Var\left[ \left( \frac{1}{n} \right)^{k} \sum_{w \in \mathcal{W}_{1,2k}}X_{w} \right]\\
&= 4k^{2}\psi_{2k}^{2}\left( 1+ O\left( \frac{k^2}{n} \right) \right) + \mathbf{E}
\end{split}
\end{equation}
where 
\begin{equation}
\mathbf{E}\le 2^{4k+2} \left( C_{3}k \right)^{C_{4}k}\sum_{t=1}^{2k-1} \left( C_{1}(4k+2) \right)^{C_{2}}\left( \frac{\left( 4k+2 \right)^{6}}{n} \right)^{2k-t}\to 0
\end{equation}
whenever $k= o\left(  \sqrt{\log n}\right)$. Similar arguments can be used to prove that 
\begin{equation}
\begin{split}
&\Cov \left[ \left( \frac{1}{n} \right)^{k} \sum_{w \in \mathcal{W}_{1,2k}}X_{w},\frac{1}{n}\Tr[\tilde{A}_{n}^{2}] \right]= 4k \psi_{2k}\left( 1+ O\left( \frac{k^2}{n}\right) \right)\\
&\Var\left[ \frac{1}{n}\Tr[\tilde{A}_{n}^{2}]  \right]= 4\left( 1+ O\left( \frac{1}{n} \right) \right).
\end{split}
\end{equation}
Plugging these estimates in \eqref{eq:secontr} we have 
\begin{equation}
\begin{split}
&\E\left[ \left( \left( \frac{1}{n} \right)^{k} \sum_{w \in \mathcal{W}_{1,2k}}X_{w}- \E\left[ \left( \frac{1}{n} \right)^{k} \sum_{w \in \mathcal{W}_{1,2k}}X_{w} \right]- k\psi_{2k}\left[ \frac{1}{n}\Tr[\tilde{A}_{n}^{2}]- \E\left[ \frac{1}{n}\Tr[\tilde{A}_{n}^{2}] \right] \right] \right)^2 \right]\\
&\le  4k^2\psi_{2k}^{2} O \left( \frac{k^2}{n} \right) + \mathbf{E}\to 0.
\end{split}
\end{equation}
Finally for each word $w\in \mathcal{W}_{3,2k}$, we have $\#V_{w}=k$. These words give nontrivial contribution to the mean. However it can be checked that
\begin{equation}
\Var\left[ \left( \frac{1}{n} \right)^{k}\sum_{w \in \mathcal{W}_{3,2k}} X_{w}\right]\le 2^{4k}\frac{\text{Poly}(k)}{n} \to 0.
\end{equation} 
In particular, one also has an explicit expression for the mean 
\begin{equation}
\begin{split}
\E\left[ \left( \frac{1}{n} \right)^{k}\sum_{w \in \mathcal{W}_{3,2k}} X_{w}  \right]= \left( 1+ O\left(  \frac{k^2}{n}\right) \right)\left[\sum_{r=3}^{k} f(2k,2r)\frac{k(r+1)}{r} + 3 f(2k,4)\frac{k}{2}\right].
\end{split}
\end{equation}
Finally the proof for the even case can be completed by following the arguments similar to \eqref{eq:quadbound} and using the fact that 
\begin{equation}\label{eq:cheidentity}
\sum_{r=1}^{k} P_{2k}[2r]r\psi_{2r}=0.
\end{equation}
One might check \citet{Banerjee2017} for a proof of \eqref{eq:cheidentity}. 
\end{enumerate}
\hfill{$\square$}
\section{Proof of Proposition \ref{prop:signdistr}}\label{sec:propsign}
We at first give the proofs of part $(1)$ and $(3)$. The proof of part $(2)$ will be given separately. \\
\textbf{Proof of part $(1)$ and $(3)$:}
We start with a very basic but fundamental observation. Note that for any $k\ge 3$
\begin{equation}
\begin{split}
C_{n,k}& = \left(\frac{1}{n}\right)^{\frac{k}{2}}\sum_{w \in \mathfrak{W}_{k+1,k}} X_{w}.
\end{split}
\end{equation}
It is easy to see that for $k\ge 2$, $\Var\left[C_{n,k}\right]\to 2k$.
The proof is completed by method of moments and Wick's formula. We at first give formal statements of these results:
At first we state the method of moments.
\begin{lemma}\label{lem:mom}
Let $(Y_{n,1},\ldots, Y_{n,l})$ be a sequence of random vectors of $l$ dimension. Then $(Y_{n,1},\ldots, Y_{n,l}) \stackrel{d}{\to} (Z_1,\ldots,Z_{l})$ if the following conditions are satisfied:
\begin{enumerate}[i)]
\item 
\begin{equation}\label{eqn:momcond}
\lim_{n \to \infty}\E[X_{n,1}\ldots X_{n,m}] 
\end{equation}
exists for any fixed $m$ and $X_{n,i} \in \{ Y_{n,1},\ldots,Y_{n,l} \}$ for $1\le i \le m$.
\item(Carleman's Condition)\cite{Carl26}
\begin{equation}
\sum_{h=1}^{\infty} \left(\lim_{n \to \infty}\E[X_{n,i}^{2h}]\right)^{-\frac{1}{2h}} =\infty ~~ \forall ~ 1\le i \le l.
\end{equation}
\end{enumerate} 
Further,  
\begin{equation}
\lim_{n \to \infty}\E[X_{n,1}\ldots X_{n,m}]= \E[X_{1}\ldots X_{m}].
\end{equation}
Here $X_{n,i} \in \{ Y_{n,1},\ldots,Y_{n,l} \}$ for $1\le i \le m$ and $X_{i}$ is the in distribution limit of $X_{n,i}$. In particular, if $X_{n,i}= Y_{n,j}$ for some $j \in \{1,\ldots,l \}$ then $X_{i}= Z_{j}$. 
\end{lemma} 
The method of moments is very well known and much useful in probability theory. We omit its proof.

\noindent
Now we state the Wick's formula for Gaussian random variables which was first proved by Isserlis(1918)\cite{I18} and later on introduced by \citet{W50} in the physics literature in 1950. 
\begin{lemma}(Wick's formula)\label{lem:wick}\citet{W50}
Let $(Y_1,\ldots, Y_{l})$ be a multivariate mean $0$ random vector of dimension $l$ with covariance matrix $\Sigma$(possibly singular). Then $((Y_1,\ldots, Y_{l}))$ is jointly Gaussian if and only if for any integer $m$ and $X_{i} \in \{ Y_1,\ldots,Y_{l} \}$ for $1\le i \le m$
\begin{equation}\label{eqn:wick}
\E[X_1\ldots X_{m}]=\left\{
\begin{array}{ll}
  \sum_{\eta} \prod_{i=1}^{\frac{m}{2}} \E[X_{\eta(i,1)}X_{\eta(i,2)}] & ~ \text{for $m$ even}\\
  0 & \text{for $m$ odd.}
\end{array}
\right.
\end{equation}
Here $\eta$ is a partition of $\{1,\ldots,m \}$ into $\frac{m}{2}$ blocks such that each block contains exactly $2$ elements and $\eta(i,j)$ denotes the $j$ th element of the $i$ th block of $\eta$ for $j=1,2$.
\end{lemma} 
The proof of aforesaid lemma is also omitted.\\
We at first verify the asymptotic CLT, then variance calculation will be given. Given $1\le k_{1} < k_{2}< \ldots < k_{l}= o\left( \sqrt{\log n} \right)$, we consider the random variables a $(Y_{n,1},\ldots, Y_{n,l})$ as follows: 
\begin{equation}
Y_{n,k_{j}}=\left\{
\begin{array}{ll}
\frac{C_{n,k_{j}}}{\sqrt{2k_{j}}} & \text{if $k_{j}\neq 2$}\\
\frac{C_{n,2}- (n-1)}{2} & \text{otherwise.}
\end{array}
\right.
\end{equation}
Now fix $m$ and take $X_{n,1},\ldots, X_{n,m}\in \{ Y_{n,1},\ldots, Y_{n,l} \}$. Let $l_{1},\ldots,l_{m}$ be the corresponding lengths of the cycles. Now observe that 
\begin{equation}\label{eq:expectm}
\begin{split}
&\E_{\mathbb{P}_{n}}\left[ X_{n,1},\ldots, X_{n,m} \right]\\
&= \left(\frac{1}{n}\right)^{\frac{\sum_{i=1}^{m}l_{i}}{2}}\prod_{i=1}^{m}\left( \frac{1}{\mathbb{I}_{l_{i}\neq 1}\sqrt{2l_{i}}+ \mathbb{I}_{l_{i}=1}} \right) \sum_{a=[w_{1}\ldots w_{m}]~|~ a ~ \text{weak CLT sentence}} \E\left[ \left( X_{w_{1}}-\E[X_{w_{1}}] \right)\ldots \left( X_{w_{l}} - \E[X_{w_{l}}]\right) \right].
\end{split}
\end{equation}
By arguments similar to the arguments given right after \eqref{eq:deom}, we have 
\begin{equation}
\E\left[ \left| \left( X_{w_{1}}-\E[X_{w_{1}}] \right)\ldots \left( X_{w_{l}} - \E[X_{w_{l}}]\right)  \right| \right]\le \left( C_{1}\left( \sum_{i=1}^{m} l_{i} \right) \right)^{C_{2}\left( \sum_{i=1}^{m} l_{i} \right)}.
\end{equation}
By Proposition A.2 in \citet{Banerjee2017}, we have $\#V_{a}< \sum_{i=1}^{m}\frac{l_{i}}{2}$ unless $a$ is a CLT sentence. Further the CLT sentences only exists if $m$ is even. Finally, we prove we can neglect the sum corresponding all the sentences which are not CLT sentences. 
\begin{equation}\label{eq:weakcltbal}
\begin{split}
&\left(\frac{1}{n}\right)^{\frac{\sum_{i=1}^{m}l_{i}}{2}}\prod_{i=1}^{m}\left( \frac{1}{\mathbb{I}_{l_{i}\neq 1}\sqrt{2l_{i}}+ \mathbb{I}_{l_{i}=1}} \right) \sum_{a=[w_{1}\ldots w_{m}]~|~ ~ \#V_{a}<  \sum_{i=1}^{m}\frac{l_{i}}{2}} \E\left[ \left|\left( X_{w_{1}}-\E[X_{w_{1}}] \right)\ldots \left( X_{w_{l}} - \E[X_{w_{l}}]\right)\right| \right]\\
&\le \left( C_{1}\left( \sum_{i=1}^{m} l_{i} \right) \right)^{C_{2}\left( \sum_{i=1}^{m} l_{i} \right)} 2^{\sum_{i=1}^{m}(l_{i}+1)} \left( C_{3} \sum_{i=1}^{m} (l_{i}+1) \right)^{C_{4}m} \sum_{t <  \sum_{i=1}^{m}\frac{l_{i}}{2}}\left( \sum_{i=1}^{m} (\l_{i}+1) \right)^{3 (\sum_{i=1}^{m}(l_{i}+1)-2t)}\left(  \frac{1}{n}\right)^{\sum_{i=1}^{m} \frac{l_{i}}{2}-t} \to 0.
\end{split}
\end{equation}
With \eqref{eq:weakcltbal} in hand, we are only left with CLT sentences. In particular for every $w_{i}$, there exists an unique $w_{j}$ such that $G_{w_{i}}$ shares an edge with $G_{w_{j}}$. On the other hand in these cases in order to get 
\begin{equation}
\E\left[ \left(X_{w_{i}}-\E\left[ X_{w_{j}} \right]\right)\left( X_{w_{j}}-\E\left[ X_{w_{j}} \right] \right) \right]\neq 0,
\end{equation} 
we need $G_{w_{i}}=G_{w_{j}}$. Further the random variables $X_{w_{i}}$ and $X_{w_{j}}$ are mutually independent if $G_{w_{i}}$ and $G_{w_{j}}$ are disjoint. Hence \eqref{eq:expectm} asymptotically satisfies Wick's formula with appropriate variance. This concludes the proof of part $(1)$ and $(3)$.\\
\textbf{Proof of part $(2)$:}
We at first give a proof for $k=2$ case. Observe that under $\mathbb{Q}_{n,\sigma}$,
\begin{equation}
\begin{split}
&\left( \frac{1}{n} \right)\sum_{i,j} A_{i,j}^{2}\\
&= \left( \frac{1}{n} \right) \sum_{i,j}\left( B_{i,j} + \frac{(2\beta)}{\sqrt{n}} \sigma_{i}\sigma_{j} \right)^2\\
&= \left( \frac{1}{n} \right) \sum_{i,j} \left( B_{i,j}^2+ \frac{4\beta \sigma_{i}\sigma_{j}}{\sqrt{n}} B_{i,j} + \frac{4\beta^2}{n} \right)\\
&= \left( \frac{1}{n} \right) \sum_{i,j} B_{i,j}^2 + \frac{4\beta \sigma_{i}\sigma_{j}}{n^{\frac{3}{2}}}\sum_{i,j} B_{i,j} + (1+ o(1)) 4\beta^2.
\end{split}
\end{equation}  
Here $B_{i,j}\sim_{i.i.d.} N(0,1)$. By CLT we have for any $\sigma$,
\begin{equation}
 \left( \frac{4\beta \sigma_{i}\sigma_{j}}{n^{\frac{3}{2}}} \right) \sum_{i,j} B_{i,j} \stackrel{p}{\to} 0.
\end{equation}  
This completes the proof for $k=2$ case.

\noindent
Now we move on to the other cases. Observe that under $\mathbb{Q}_{n,\sigma}$
\begin{equation}\label{eq:expandcyclealt}
\begin{split}
C_{n,k}&= \left( \frac{1}{\sqrt{n}} \right)^{k}\sum_{w \in \mathfrak{W}_{k+1,k}} \prod_{j=0}^{k-1}A_{i_{j},i_{j+1}}\\
&= \left( \frac{1}{\sqrt{n}} \right)^{k}\sum_{w \in \mathfrak{W}_{k+1,k}} \prod_{j=0}^{k-1} \left( B_{i_{j},i_{j+1}} + \frac{2\beta\sigma_{i_{j}}\sigma_{i_{j+1}}}{\sqrt{n}} \right)\\
&= \left( \frac{1}{\sqrt{n}} \right)^{k}\sum_{w \in \mathfrak{W}_{k+1,k}}\prod_{j=0}^{k-1}B_{i_{j},i_{j+1}} + \sum_{w \in \mathfrak{W}_{k+1,k}} V_{n,k,w} + \left( 1+ O \left( \frac{k^2}{n} \right) \right)(2\beta)^{k}. 
\end{split}
\end{equation}
Here $V_{n,k,w}$ is obtained by expanding the product in \eqref{eq:expandcyclealt} and taking all the residual terms in the product. Observe that  $w=(i_{0},i_{1},\ldots,i_{k})$, 
\begin{equation}
\begin{split}
V_{n,k,w}&= \left( \frac{1}{n} \right)^{\frac{k}{2}} \sum_{ \emptyset \subsetneq E_{f}\subsetneq E_{w}} \prod_{e \in E_{f}} \sigma_{e} (\frac{2\beta}{\sqrt{n}}) \prod_{e \in E_{w} \backslash E_{f}} B_{e} 
\end{split}
\end{equation}
Here for any edge $\{i,j\}$, $B_{e}=B_{i,j}$ and $\sigma_{e}={\sigma_{i}\sigma_{j}}$. Observe that for any $\sigma$,  $\E[V_{n,k,w}]=0$. We now prove that 
\begin{equation}
\E\left[\left( \sum_{w \in \mathfrak{W}_{k+1,k}}V_{n,k,w} \right)^{2}\right] \to 0.
\end{equation}
We have that 
\begin{equation}
\begin{split}
& \E\left[\left( \sum_{w \in \mathfrak{W}_{k+1,k}}V_{n,k,w} \right)^{2}\right]\\
&= \left( \frac{1}{n}\right)^{k} \sum_{w,x \in \mathfrak{W}_{k+1,k}} \E\left[ V_{n,k,w} V_{n,k,x}\right]. 
\end{split}
\end{equation}
We now find an upper bound to $\E\left[ V_{n,k,w} V_{n,k,x} \right]$. At first fix any word $w$ and the set $\emptyset \subsetneq E_{f} \subsetneq E_{w}$ and consider all the words $x$ such that $E_{{w}}\cap E_{x}= E_{w} \backslash E_{f}$. As every edge in $G_{w}$ and $G_{x}$ appear exactly once,
\begin{equation}
\begin{split}
& \E\left[ V_{n,k,w} V_{n,k,x} \right]\\
&= \left( \frac{1}{n} \right)^{k} \sum_{E_w\backslash E' \subset E_{w} \backslash E_{f}} \prod_{e \in E'}\left(\pm \frac{4\beta^2}{n}\right)\E\left[\prod_{e \in E_{w}\backslash E'} B_{e}^2\right]\\
&\le \left( \frac{1}{n} \right)^{k} \sum_{E_w\backslash E' \subset E_{w} \backslash E_{f}}\left( \frac{4\beta^2}{n} \right)^{\#E'}\\
&\le \left( \frac{1}{n} \right)^{k+ \#E_{f}}2^{k}.   
\end{split}
\end{equation}
The last inequality holds since $\#E'\ge \#E_{f}$ and  $\#(E_w\backslash E' \subset E_{w} \backslash E_{f})\le 2^{k}$.

Observe that the graph corresponding to the edges $E_{w} \backslash E_{f}$ is a disjoint collection of straight lines. Let the number of such straight lines be $\zeta$. Obviously $\zeta \le \#(E_{w} \backslash E_{f})$.  The number of ways these $\zeta$ components can be placed in $x$ is bounded by $k^{\zeta}\le k^{\#(E_{w} \backslash E_{f})}$ and all other nodes in $x$ can be chosen freely. So there are at most $n^{k-\#V_{E_{w}\backslash E_f}}k^{\#(E_{w} \backslash E_{f})}$ choices of such $x$. Here $V_{E_{w}\backslash E_{f}}$ is the set of vertices of the graph corresponding to $(E_{w} \backslash E_{f})$. Observe that, whenever $k>\#E_{f}>0$, $E_{w}\backslash E_{f}$ is a forest so $$\#V_{E_{w}\backslash E_{f}}\ge \#(E_{w} \backslash E_{f})+1 \Leftrightarrow k-\#V_{E_{w}\backslash E_f} \le \#E_{f}-1.$$
As a consequence, 
\begin{equation}\label{eqn:booboo}
\sum_{x~|~ E_{{w}}\cap E_{x}= E_{w} \backslash E_{f}}\cov(V_{n,k,w}, V_{n,k,x})\le (2)^{k}\frac{1}{n^{k+\#E_{f}}}n^{\#E_{f}-1}k^{\#(E_{w} \backslash E_{f})}\le  (2)^{k} \frac{1}{n^{k+1}}k^{k} .
\end{equation}
The right hand side of \eqref{eqn:booboo} does not depend on $E_{f}$
and there are at most $2^{k}$ nonempty subsets $E_{f}$ of $E^{w}$.
So
\begin{equation}
\sum_{x}\cov(V_{n,k,w}, V_{n,k,x}) \le (4)^{k} k^{k}\frac{1}{n^{k+1}}.
\end{equation} 
Finally there are at most $n^{k}$ many $w$. So
\begin{equation}\label{eqn:final}
\sum_{w}\sum_{x}\cov(V_{n,k,w}, V_{n,k,x})\le (4)^{k}k^{k}\frac{1}{n}.
\end{equation}
Now we use the fact $k=o(\sqrt{\log(n)})$. In this case 
\begin{equation}
k\log(4)+k\log(k)\le \sqrt{log(n)}\log(\sqrt{\log n}) = o(log(n)) \Leftrightarrow (4)^{k}k^{k}=o(n).
\end{equation}
This concludes the proof. \hfill{$\square$}

\noindent 
\textbf{Acknowledgment:} The author acknowledges Prof. Jinho Baik for suggesting this problem to him while he was visiting U Michigan for a summer school. The author also thanks Prof. Wei Kuo Chen and the referees for comments.
\bibliography{PAR_SPI}
\end{document}